\documentclass[12pt,a4paper]{article}
\usepackage{amsmath,amsthm}
\usepackage{amsfonts}
\usepackage{amssymb}
\usepackage{hyperref}
\usepackage[numbers]{natbib}
\usepackage{authblk}
\newtheorem{theorem}{Theorem}[section]

\newtheorem{lemma}[theorem]{Lemma}

\newcommand{\cw}{\overset{d}{\rightarrow}}
\newcommand{\cp}{\overset{P}{\rightarrow}}
\newcommand{\bbet}{\boldsymbol{\beta}}
\newcommand{\hbbet}{\hat{\boldsymbol{\beta}}}

\begin{document}

\title{Asymptotic Statistical Properties of Redescending M-estimators in Linear Models with Increasing Dimension}

\author[1]{Ezequiel Smucler}

\affil[1]{Instituto de Calculo, Universidad de Buenos Aires - CONICET}

\newcommand{\Addresses}{{% additional braces for segregating \footnotesize
  \bigskip
  \footnotesize
\textsc{\\
Instituto de C\'alculo \\ Universidad de Buenos Aires \\ Ciudad Universitaria, Pabell\'on 2\\ Buenos Aires 1426\\ Argentina \\ \url{esmucler@ic.fcen.uba.ar}}\par\nopagebreak
 
}}

\date{}

\maketitle

\begin{abstract}
This paper deals with the asymptotic statistical properties of a class of redescending M-estimators in linear models with increasing dimension. This class is wide enough to include popular high breakdown point estimators such as S-estimators and MM-estimators, which were not covered by existing results in the literature. We prove consistency assuming only that $p/n \rightarrow 0$ and asymptotic normality essentially if $p^{3}/n \rightarrow 0$, where $p$ is the number of covariates and $n$ is the sample size.
\end{abstract}
%
%\begin{keyword}[class=MSC]
%\kwd[Primary ]{62F35}
%\kwd[; secondary ]{62F12, 62J05}
%\end{keyword}
%% robust
%% asymp propert, linear reg

\section{Introduction}
\label{sec:Intro}
The growing number of statistical problems with a large
number of parameters has motivated the study of the asymptotic properties of estimators for statistical
models with a number of parameters that diverges with the sample size. For the case of linear regression, consider a sequence of regression models
\begin{equation}
y_{i,n} = \mathbf{x}_{i,n}^{T} \boldsymbol{\beta}_{0,n} + u_{i,n},\:\:\:
1\leq i \leq n  \nonumber
\end{equation}
where $y_{i,n} \in \mathbb{R}$, $\mathbf{x}_{i,n} \in \mathbb{R}^{p_n}$ is a vector of fixed predictor variables, $%
\boldsymbol{\beta}_{0,n} \in \mathbb{R}^{p_n}$ is to be estimated and $u_{i,n}$ are i.i.d. random variables defined in a common probability space with distribution function $F_0$. We consider the case in which $p_n$ may tend to infinity with $n$ at a certain rate. To unburden the notation, we will drop the $n$ subscript from $y_{i,n}, \mathbf{x}_{i,n}$, $\boldsymbol{\beta}_{0,n}$, $p_n$ and $
u_{i,n}$.

It is well known that the Least Squares estimator of $\bbet_0$ is not robust. This fact has led to the development of robust estimators.
A general framework for estimation in the linear model is provided by M-estimators. The notion of an M-estimator was first introduced in the landmark paper \citep{Huber64} for the case of the estimation of a location parameter and extended to the linear model in \citep{Huber73}. Given a suitably chosen loss function $\rho$, the corresponding regression M-estimator is defined by
\begin{equation}
\hbbet=\arg\min_{\bbet \in \mathbb{R}^{p}} \sum\limits_{i=1}^{n} \rho\left(\frac{r_{i}(\bbet)}{
s_n} \right),
\label{Eq:M}
\end{equation}
where $r_i(\bbet)=y_i - \mathbf{x}_{i}^{T}\bbet$ and $s_n$ is an estimate of scale of the residuals that may be estimated a priori or simultaneously. For example, $s_n$ could be the median of the absolute values of the residuals of some initial regression estimator. 

For the case of a convex and differentiable loss function, \eqref{Eq:M} is essentially equivalent to
\begin{equation}
\sum\limits_{i=1}^{n} \psi\left(\frac{r_{i}(\hbbet)}{
s_n} \right)\mathbf{x}_{i}=\mathbf{0},
\label{Eq:Mzero}
\end{equation}
where $\psi = \rho^{\prime}$; see Section 7.3 of \citep{Huber81} and Section 4.4 of \citep{Libro}. In this case, the resulting M-estimator is called a monotone regression M-estimator. When $\psi$ tends to zero at infinity the resulting estimator is called a redescending regression M-estimator and in this case some solutions of \eqref{Eq:Mzero} may not correspond to solutions of \eqref{Eq:M}.

The robustness of an estimator is measured by its stability when a small fraction of the observations is arbitrarily replaced by outliers
that may not follow the assumed model. A robust estimator should not be
much affected by a small fraction of outliers. A popular quantitative measure of an estimator's robustness, introduced
by \citep{Donoho-Breakdown}, is the finite-sample replacement breakdown point. 
Very loosely speaking, the finite-sample replacement breakdown point of an estimator is the maximum fraction of outliers that the estimator may tolerate without losing all meaning.
For a regression estimator, this measure is defined as follows. Given a sample
$\mathbf{z}_{i}=(\mathbf{x}_{i}^{\text{T}},y_{i})$, $i=1,...,n$, let
$\mathbf{Z}=\{\mathbf{z}_{1},...,\mathbf{z}_{n}\}$ and let $\hat
{\bbet}(\mathbf{Z})$ note the regression estimator $\hbbet$ computed in $\mathbf{Z}$. The finite-sample
replacement breakdown point of $\hat{\bbet}$ is then defined as
$FBP(\hat{\bbet})=m^{\ast}/n$, where $
m^{\ast}=\max \left\{m\geq0:\hat{\bbet}(\mathbf{Z}_{m})\text{ is
bounded for all }\mathbf{Z}_{m}\in\mathcal{Z}_{m}\right\}$
and $\mathcal{Z}_{m}$ is the set of all datasets with at least $n-m$ elements
in common with $\mathbf{Z}$. It can be shown that any regression equivariant estimator has a breakdown point of at most $[(n-p)/2]/n$, which is approximately 1/2 for $p\ll n$. See, for example, Section 5.4.1 of \citep{Libro}.

It can be shown that, if we only entertain the possibility of outliers in the response variable, monotone M-estimators defined by \eqref{Eq:Mzero} with a bounded $\psi$ may have a high breakdown point. This holds, for example, for the cases of one-way or two-way ANOVA designs; see Section 4.6 of \citep{Libro}. However, if outliers in the predictor variables are a possibility, the breakdown point of monotone regression M-estimators is zero; see Section 5.16.1 of \citep{Libro}. Moreover, monotone regression M-estimators may be highly inefficient when the errors are heavy tailed. Excellent discussions of the robustness properties of regression M-estimators can be found in \citep{Huber81}, \citep{HampelLibro} and \citep{Libro}.

A brief history of the study of the asymptotic properties of estimators for linear regression models with a diverging number of parameters when $p/n \rightarrow 0$ goes as follows. To the best of our knowledge, the first analysis of this problem appears in \citep{Huber73}. In \citep{Huber73}, Huber studied the asymptotic properties of monotone regression M-estimators defined without using an estimate of scale. Motivated by problems in X-ray crystallography, Huber proposed to study the properties of these estimators when $p=p_n \rightarrow \infty$. He proved the asymptotic normality of linear contrasts of these estimators when $p^{3}/n \rightarrow 0$. This result was improved by \citep{Maronna79}, who, under essentially the same hypothesis as \citep{Huber73}, proved the asymptotic normality of linear contrasts requiring only $p^{5/2}/n \rightarrow 0$ and also the $\sqrt{n/p}$-consistency assuming $p^{2}/n \rightarrow 0$. \citep{Maronna79} also provided analogous results for the case of monotone M-estimators defined using an estimate of scale. In \citep{Portnoy} and \citep{Portnoy85}, Portnoy studied the asymptotic properties of the solutions of M-estimating equations, \eqref{Eq:Mzero}, without including an estimate of scale and where the loss function is not necessarily convex. For the case of a convex loss function, Portnoy proved the $\sqrt{n/p}$-consistency of the estimators and the asymptotic normality of linear contrasts, requiring that $(p \log p) /n \rightarrow 0$ and $(p \log n)^{3/2} /n \rightarrow 0$ respectively. For the case of a non-convex loss function, Portnoy showed that there exists a solution of \eqref{Eq:Mzero} that is consistent, however, this solution may not be a solution of \eqref{Eq:M}. \citep{Mammen88} obtained asymptotic expansions for the solutions of \eqref{Eq:Mzero}, without including an estimate of scale and when the loss function is convex, assuming only $p^{3/2} \log n /n \rightarrow 0$. For the cases in which the loss function is not convex or a scale is estimated simultaneously, \citep{Mammen88} showed that there exists a solution of \eqref{Eq:Mzero} for which an analogous asymptotic expansion holds, however, this solution may not be a solution of \eqref{Eq:M}. \citep{Welsh} obtained results under more relaxed assumptions on the regularity of the loss function $\rho$ but under more stringent conditions on the rate of growth of $p$. \citep{BaiI} and \citep{BaiII} further improved the aforementioned results by relaxing the regularity conditions imposed on $\rho$ or the rate of growth of $p$. For example, for the case of a sufficiently smooth and convex loss function, they proved the consistency and asymptotic normality of M-estimators assuming $p/n \rightarrow 0$ and $p^2/n \rightarrow 0$ respectively. \citep{HE2000120} studied M-estimators of general parametric models with increasing dimension. For the case of a linear model, they obtained results similar to those of \citep{Welsh}.

More recently, \citep{Karoui}, \citep{Karoui2}, \citep{Donoho-Montanari}, \citep{Donoho-MontanariH} and \citep{Ritov} have studied the asymptotic properties of monotone M-estimators when $p/n\rightarrow m \in (0,1)$.

None of the aforementioned results are directly applicable to M-estimators defined using a bounded loss function or to high-breakdown point estimators such as S-estimators (\citep{S}) or MM-estimators \citep{MM87}; the precise definition of these estimators will be stated in the following section, since it is somewhat technical. \citep{Davies} proved the consistency of regression S-estimators assuming $(p \log n) /n \rightarrow 0$.

In this paper, we prove consistency and asymptotic normality results for a class of redescending M-estimators that is wide enough to include both S and MM-estimators. More precisely, we prove the consistency of the estimators under very general assumptions and requiring only that $p/n \rightarrow 0$ and we prove their asymptotic normality essentialy when $p^{3}/n \rightarrow 0$. Since S and MM-estimators are two of the most popular robust regression methods, and since the asymptotic results in linear models with increasing dimension that are available in the literature, except for the consistency result of \citep{Davies}, are not applicable to them, we consider our results to be a valuable contribution to the theory of robust regression.

The rest of this paper is organized as follows. In Section \ref{Section-High} we state the definitions of S and MM-estimators and show how they can be cast as redescending M-estimators defined using a suitably chosen scale estimate. In Section \ref{Section-DefAs} we state the assumptions needed to prove our results. Furthermore, we compare our assumptions with those previously
considered in the literature. In Section \ref{Section-Res} we state and prove our main results. Section \ref{Section-App} is a technical appendix containing technical lemmas that are needed in the proofs of our main results.

\section{High breakdown point regression estimators}
\label{Section-High}

Throughout this paper, we will say that $\rho$ is a $\rho$-function if: $\rho$ is even and continuous, $\rho(x)$ is a nondecreasing function of $\vert x \vert$,  $\rho(0)=0$, if $\rho(v)< \lim_{x\rightarrow \infty}\rho(x)$ and $0\leq u<v$ then $\rho(u)<\rho(v)$, if $\rho$ is bounded, $\lim_{x\rightarrow \infty}\rho(x)=1$. If $\rho$ is a differentiable $\rho$-function we will set $\psi=\rho^{\prime}$.

A popular family of $\rho$-functions is Tukey's Bisquare family of loss functions, given by $
\rho^{B}_{c}(x)=1-\left(1-\left( x/c \right)^2\right)^3 I\lbrace \vert x \vert \leq c\rbrace$,
where $c>0$ is some tuning constant. Note that $\rho^{B}_{c}$ is bounded, moreover, if $\vert x \vert \geq c$, then $\rho^{B}_{c}(x)=1$.

To define S and MM-estimators, we first need to define M-estimators of scale.
Let $\rho_0$ be a bounded $\rho$-function. Given a sample $\mathbf{u}=(u_{1},...,u_{n})$ and $0<b<1$ the corresponding
M-estimate of scale $s_{n}(\mathbf{u})$ is defined, \citep{Huber81}, by
\begin{equation}
s_{n}(\mathbf{u})=\inf\left\{  s>0:\frac{1}{n}\sum\limits_{i=1}^{n}%
\rho_0\left(  \frac{u_i}{s}\right)  \leq b\right\}  . \label{Eq:defMscale}
\end{equation}
It is easy to prove that $s_{n}(\mathbf{u})>0$ if and only if $\#\{i:u_{i}%
=0\}<(1-b)n,$ and in this case%
\begin{equation}
\frac{1}{n}\sum\limits_{i=1}^{n}\rho_0\left(  \frac{u_{i}}{s_{n}(\mathbf{u}%
)}\right)  =b. \nonumber%
\end{equation}

S-estimators, introduced in \citep{S}, are regression estimators that can be tuned to have a high breakdown point. They are defined by
\begin{align*}
\hbbet_{S}=\arg\min_{\bbet\in\mathbb{R}^{p}} s_{n}(\mathbf{r}(\bbet))
\end{align*}
where $\mathbf{r}(\bbet)=\left(r_{1}(\bbet),\dots,r_{n}(\bbet) \right)$ and $s_n()$ is an M-estimator of scale. It is easy to verify that S-estimators are scale and regression equivariant.
Let $\hat{s}_n=s_{n}(\mathbf{r}(\hbbet_S))$ and let $\rho_0$ be the $\rho$-function used to define $s_n()$. Then, S-estimators satisfy
\begin{align*}
\hbbet_{S}=\arg\min_{\bbet\in\mathbb{R}^{p}} \sum\limits_{i=1}^{n} \rho_0 \left( \frac{r_i(\bbet)}{\hat{s}_{n}}\right),
\end{align*}
see Section 5.6.1 of \citep{Libro}. Hence, S-estimators are M-estimators in the sense of \eqref{Eq:M}, where the loss function $\rho$ is bounded and the scale is estimated simultaneously. In practice, $\rho_0$ is usually chosen so that it satisfies $\rho_0(x)=1$ if $\vert x \vert\geq m$ for some $m$. For example, $\rho_0$ could be Tukey's Bisquare loss. For the case of fixed $p$, the asymptotic distribution of
regression S-estimators was derived, under very general conditions, by \citep{Fasano} for the case of random predictors and by \citep{Davies} for the case of fixed predictors. S-estimators can always be tuned so as to attain the maximum possible
finite-sample replacement breakdown point for regression equivariant estimators; see Section 5.6.1 of \citep{Libro}. However, S-estimators cannot combine a high breakdown point with a high
efficiency at the normal distribution, see \citep{Hossjer}.
 
MM-estimators, introduced in \citep{MM87}, are regression estimators that can
be tuned to attain both a high breakdown point and an arbitrarily high asymptotic efficiency at
the normal distribution. Suppose $\hbbet_1$ is a highly robust, but not necessarily highly efficient, initial estimator. In practice, $\hbbet_1$ will usually be an S-estimator. Let $s_n()$ be an M-estimator of scale defined using a bounded $\rho$-function $\rho_0$ and $b$. Let $\rho_1$ be another $\rho$-function that satisfies $\rho_1\leq \rho_0$. Then the MM-estimator is defined by
\begin{align*}
\hbbet_{MM}=\arg\min_{\bbet\in\mathbb{R}^{p}}  \sum\limits_{i=1}^{n} \rho_1\left(\frac{r_{i}(\bbet)}{
s_n(\mathbf{r}(\hbbet_1))} \right).
\end{align*}
Note that MM-estimators are M-estimators, as in \eqref{Eq:M}, defined using a bounded loss function and a preliminary estimate of scale. MM-estimators are scale and regression equivariant whenever $\hbbet_1$ satisfies these properties.  We note that the original definition of MM-estimators is actually more general, but for technical convenience we will work with this definition. 
\citep{MM87} proved the strong consistency and asymptotic normality of MM-estimators for the case of fixed $p$ and random predictor variables.

MM-estimators are the robust regression estimators of choice of \citep{Libro}. The authors recommend the use of an S-estimator with maximal breakdown point as the initial estimator when computing MM-estimators. The resulting MM-estimator will also have maximal breakdown point. They recommend taking $\rho_0 = \rho^{B}_{c_{0}}$ and $\rho_1 = \rho^{B}_{c_{1}}$ with $c_1\geq c_0$, where $\rho^{B}_{c}$ is Tukey's Bisquare loss and $c_1$ and $c_0$ are suitably chosen tuning constants.
The fact that MM-estimators can be tuned to attain both a high breakdown point and an arbitrarily high asymptotic efficiency at the normal distribution has made them one of the most popular alternatives robust regression has to offer.

\section{Definitions and assumptions}
\label{Section-DefAs}

In what follows, we will consider 
\begin{align*}
L_n(\bbet) = \sum\limits_{i=1}^{n} \rho_1\left(\frac{r_{i}(\bbet)}{s_n} \right),
\end{align*}
where $\rho_1$ is a bounded $\rho$-function, $s_n$ is a positive random variable defined in the same probability space as the errors $u_i$ and 
\begin{equation}
\hbbet=\arg\min_{\bbet \in \mathbb{R}^{p}} L_n(\bbet).
\label{Eq:M2}
\end{equation}
We will let $s_{n}^{S}()$ be the M-estimator of scale defined using a bounded $\rho$-function $\rho_0$ and $b\in(0,1)$.
We will assume that $s_n$, the random variable used to standardize the residuals in \eqref{Eq:M2}, converges in probability to some value $s_0>0$. For example, according to Lemma \ref{lemma-conv-sca} and the comments following it, one may take $s_n=s_{n}^{S}(\mathbf{r}(%
\hat{\boldsymbol{\beta}}_{S}))$, where $\hbbet_{S}$ is the S-estimator that minimizes the M-estimator of residual scale $s_{n}^{S}(\mathbf{r}(\bbet))$. Note that when $\rho_1=\rho_0$ and $s_n=s_{n}^{S}(\mathbf{r}(\hat{\boldsymbol{\beta}}_{S}))$, $\hbbet$ is an S-estimator, whereas when $\rho_1 \leq \rho_0$ and $s_n=s_{n}^{S}(\mathbf{r}(\hat{\boldsymbol{\beta}}_{S}))$, $\hbbet$ is an MM-estimator.

We introduce some notation. Let $\rho_{1,n}$ and $\rho_{2,n}$ stand for the smallest and largest
eigenvalues of $\boldsymbol{\Sigma}_n = (1/n) \sum\limits_{i=1}^{n} \mathbf{x}_{i}%
\mathbf{x}_{i}^{T}$. We will assume that $\boldsymbol{\Sigma}_n$ is non-singular for all $n$.
For $0<\alpha < 1$, let 
\begin{equation}
\eta_n(\alpha)= \min_{\mathcal{A}\subset \lbrace1,...,n\rbrace, \#\mathcal{A}%
=[n\alpha]} \min_{\Vert \theta \Vert=1} \max_{i\in\mathcal{A}} \vert \mathbf{%
x}_{i}^{T} \theta \vert.
\nonumber
\end{equation}
For $\mathbf{z}\in\mathbb{R}^{p}$ and $c>0$, let $
I(\mathbf{z},c) = \left\lbrace i = 1, \dots, n : \left\vert \mathbf{x}_{i}^{T} \mathbf{z} \right \vert \leq c \right\rbrace$,
let $\mathcal{B}(\delta)$ be the ball in $\mathbb{R}^{p}$ centered at zero with radius $\delta$ and let $\mathcal{S}^{*}$ be the sphere centered at zero with radius 1. 

We will need the following
assumptions:

\begin{itemize}
\item[R0.] $\rho_0$ is a bounded $\rho$-function and, for some $m>0$, $\rho_0(u)=1$ if $\vert u \vert \geq m$.

\item[R1.] $\rho_1$ is a continuously differentiable, bounded $\rho$-function. Let $\psi_1$ be the derivative of $\rho_1$. Then $\psi_1(t)$ and $t\psi_1(t)$ are bounded.

\item[R2.] $\rho_1$ is a three times continuously differentiable, bounded $\rho$-function. Let $\psi_1$ be the derivative of $\rho_1$. Then $\psi_1(t), \psi_{1}^{\prime}(t), \psi_{1}^{\prime\prime}(t),t\psi_1(t), t\psi_{1}^{\prime}(t)$ and  $t\psi_{1}^{\prime\prime}(t)$ are bounded. Also, $\mathbb{E}\psi_{1}^{\prime}\left( u/s_0 \right)>0$, where $s_0$ is the limit in probability of $s_n$.

\item[F0.] $F_{0}$ has a density, $f_0$, that is even, a
monotone decreasing function of $|u|$ and a strictly decreasing function of $%
|u|$ in a neighbourhood of 0.

\item[X0.] $p<[n(1-b)]$ for all $n$, where $b$ is the constant used in \eqref{Eq:defMscale}.

\item[X1.] \begin{enumerate} \item[a)] There exists a constant $M>0$ such that $ (1/n) \sum\limits_{i=1}^{n}  \Vert \mathbf{x}_{i} \Vert^2 \leq p M$ for all $n$. 
\item[b)] There exists a constant $B>0$ such that $\max_{i \leq n} \Vert \mathbf{x}_{i} \Vert \leq B n$ for all $n$.
\end{enumerate}

\item[X2.] $\tau = \sup_{n} \rho_{2,n} <\infty$.

\item[X3.] For some $0<\alpha<1$, $\liminf \eta_{n}(\alpha)>0$.

\item[X4.] For any $c>0$ there are constants $a>0$, $\delta>0$ and $C>0$ such that for all $\bbet \in \mathcal{B}(\delta)$, $\mathbf{z} \in \mathcal{S}^{*}$, and $n$, $
\sum_{i \in J} \left( \mathbf{x}_{i}^{T} \mathbf{z} \right)^{2} \geq a n$,
where $J= I(\bbet,c)\cap I(\mathbf{z},C)$.

\item[X5.] For any $c>0$ and $\varepsilon>0$ there are constants $\delta^{'}>0$ and $C>0$ such that for all $\bbet \in \mathcal{B}(\delta^{'})$, $\mathbf{z} \in \mathcal{S}^{*}$, and $n$, $
\sum_{i \notin J} \left( \mathbf{x}_{i}^{T} \mathbf{z} \right)^{2} \leq \varepsilon n,$
where $J= I(\bbet,c)\cap I(\mathbf{z},C)$.

\item[X6.]   $\max_{i\leq n} \Vert \mathbf{x}_{i} \Vert^{2} = o(n /p^{2})$.

\end{itemize}

Conditions [R0] and [R1] are satisfied by, for example, Tukey's Bisquare loss function. Condition [R2] is a strengthening of condition [R1]. It is satisfied by, for example, the exponential squared loss $\rho(x)=1-\exp(-x^2)$ and $\rho(x)=1-\left( 1- x^{2} \right)^{4} I\left\lbrace \vert x \vert \leq 1 \right\rbrace$, which is similar to Tukey's Bisquare loss. 

Note that condition [F0] does not require finite moments from $F_0$. Thus, extremely heavy tailed error distributions, such as Cauchy's distribution, can be easily seen to satisfy it.

Condition [X0] is needed in the proof of the consistency of the scale estimate provided by the S-estimator.
To prove the consistency of the regression estimators we will need $p/n \rightarrow 0$.
To obtain the rate of consistency of the estimators we will need $(p \log n)/n \rightarrow 0$.  Note that $(p \log n)/n \rightarrow 0$ is no stronger than $(p \log p)/n \rightarrow 0$, paraphrasing \citep{Portnoy}: if $p\leq \sqrt{n}$, $(p \log n)/n \leq (\log n) / \sqrt{n}\rightarrow 0$; while if $p\geq \sqrt{n}$, $(p \log n)/n \leq (2 p \log p)/n$.  

[X1] a) holds when the covariates are standardized. [X1] b) appears in \citep{Portnoy} and holds, for example, if all the covariates are bounded and $p/n^2 \rightarrow 0$. On the other hand, suppose that $\mathbf{x}_{i}$, $i=1,\dots,n$ are independent and identically distributed random vectors in $\mathbb{R}^{p}$ such that for some $C$, $\mathbb{E}x_{i,j}^{2}\leq C$ for all $i,j$ and $n$. Then, [X1] holds in probability if $p/n \rightarrow 0$; see Section 4 of \citep{Portnoy}. [X2] appears in, for example, \citep{Portnoy85} and \citep{Welsh}. See also \citep{BaiI}. 

The function $\eta_n(\alpha)$ that appears in [X3] was introduced in \citep{Davies}. It measures in some sense the worst possible
conditioning of any subset of size $[n\alpha]$ of the carriers.
Suppose that $\mathbf{x}_{i}$, $i=1,\dots,n$ are independent and identically distributed random vectors in $\mathbb{R}^{p}$ such that there exists $\eta_1, \eta_2$ with $0<\eta_1, \eta_2<1$ such that, for all $n$, $
\sup_{\Vert \boldsymbol{\theta} \Vert =1} \mathbb{P} \left( \vert \mathbf{x}^{T} \boldsymbol{\theta} \vert < \eta_1 \right) < 1 - \eta_2$. We will show that in this case, [X3] holds in probability. 
Note that  $
\sup_{\Vert \boldsymbol{\theta} \Vert =1} \mathbb{P} \left( \vert \mathbf{x}^{T} \boldsymbol{\theta} \vert < \eta_1 \right) < 1 - \eta_2$ holds, for some $0<\eta_1, \eta_2<1$ and all $n$, for example, if $\mathbf{x}_{i} \sim N_{p}(\mathbf{0},\mathbf{M}_n)$ and there exists some $\kappa>0$ such that the smallest eigenvalue of $\mathbf{M}_n$ is bounded below by $\kappa$ for all $n$. It is easy to show, using maximal inequalities such as those of Theorem \ref{Theo-Pollard} in the Appendix, that if $p/n \rightarrow 0$,
\begin{align*}
\sup_{\Vert \boldsymbol{\theta} \Vert =1} \left\vert \frac{1}{n}\sum\limits_{i=1}^{n} I\left\lbrace \vert \mathbf{x}_{i}^{T} \boldsymbol{\theta} \vert< \eta_1\right\rbrace -   \mathbb{P} \left( \vert \mathbf{x}^{T} \boldsymbol{\theta} \vert < \eta_1 \right)  \right\vert \cp 0.
\end{align*}
Hence, with arbitrarily high probability, for large enough $n$, 
\begin{align*}
\sup_{\Vert \boldsymbol{\theta} \Vert =1} \frac{1}{n}\sum\limits_{i=1}^{n} I\left\lbrace \vert \mathbf{x}_{i}^{T} \boldsymbol{\theta} \vert< \eta_1\right\rbrace < \sup_{\Vert \boldsymbol{\theta} \Vert =1} \mathbb{P} \left( \vert \mathbf{x}^{T} \boldsymbol{\theta} \vert < \eta_1 \right) + \eta_2/2 < 1- \eta_2/2.
\end{align*}
In this case, for any $\alpha$ such that $1-\eta_2/2<\alpha<1$, for large enough $n$ it follows that for all $\boldsymbol{\theta} $ with $\Vert \boldsymbol{\theta} \Vert =1$ and all subsets $\mathcal{A}$ of $\lbrace 1,\dots,n \rbrace$ with $\#\mathcal{A}=[n\alpha]$ there exists $i\in\mathcal{A}$ such that $\vert \mathbf{x}_{i}^{T}\boldsymbol{\theta} \vert \geq \eta_1$, which implies $\eta_{n}(\alpha)\geq \eta_1$. 

For $\mathcal{A}\subset \lbrace1,...,n\rbrace$ with $\#\mathcal{A}%
=[n\alpha]$ let $
\Sigma(\mathcal{A})= (1/[n\alpha]) \sum_{i\in \mathcal{A}} \mathbf{x}_{i}\mathbf{x}_{i}^{T}$.
Let $\rho_{1,n}(\mathcal{A})$  be the smallest eigenvalue of $\Sigma(\mathcal{A})$.
Take $\boldsymbol{\theta}$ with $\Vert \boldsymbol{\theta} \Vert=1$. Then $
\boldsymbol{\theta}^{T}\Sigma(\mathcal{A})\boldsymbol{\theta}\leq \max_{i\in\mathcal{A}} \vert \mathbf{x}_{i}^{T} \boldsymbol{\theta}\vert^{2}$.
Hence $
\rho_{1,n}(\mathcal{A}) \leq \min_{\Vert \boldsymbol{\theta}\Vert=1} \max_{i \in \mathcal{A}}\vert \mathbf{x}_{i}^{T} \boldsymbol{\theta}\vert^{2}$,
which implies that
\begin{align*}
\min_{\mathcal{A}\subset \lbrace1,...,n\rbrace, \#\mathcal{A}%
=[n\alpha]} \rho_{1,n}(\mathcal{A}) \leq \eta_n(\alpha)^{2}.
\end{align*}
It follows that $\liminf \eta_{n} (\alpha)>0$ holds if the smallest eigenvalues of the covariance matrices formed from any subsample of size $[n \alpha]$ are uniformly bounded away from zero. 
See also Examples 1, 2 and 3 of \citep{Davies}.
The following lemma, the proof of which can be found in the Appendix, gives necesary conditions for $\liminf \eta_{n} (\alpha)>0$ to hold.
\begin{lemma}
\label{lemma:cond_davies}
Assume [X1] a) holds. Then, if $\liminf \eta_{n} (\alpha)>0$ for some $0<\alpha<1$, there exists positive numbers $\eta_1,\eta_2$ and $n_0$ such that
\begin{align*}
\frac{1}{n}\sum_{i=1}^{n}\mathbf{x}_i \mathbf{x}_{i}^{T} I\lbrace \Vert \mathbf{x}_i \Vert < \eta_1 \sqrt{p} \rbrace - \eta_2 \mathbf{I}_p
\end{align*}
is positive definite for all $n\geq n_0$.
\end{lemma}
Note that if [X1] and [X3] hold, by Lemma \ref{lemma:cond_davies} we have that $\inf_{n} \rho_{1,n}>0$.

[X4] and [X5] were introduced in \citep{Portnoy} where they appear as X1 and X2. \citep{Portnoy} showed that these conditions hold in probability if the covariates are sampled from an appropriate distribution in $\mathbb{R}^{p}$, such as a scale mixture of standard multivariate normals, and $(p \log n)/n \rightarrow 0$.
[X4] and [X5] are used in Lemma \ref{lemma-mod-portnoy}, a result that is needed in the proof of the rate of convergence of the estimators. The aforementioned lemma shows that, very loosely speaking, $L_n(\bbet)$ is convex in a neighbourhood of the true regression parameter with probability tending to one.

[X6] is needed in the proof of the asymptotic normality of the estimators. It holds, for example, if the covariates are bounded and $p^{3}/n \rightarrow 0 $. This is the rate of growth of $p$ allowed by the asymptotic normality result of \citep{Huber73}.

\section{Results}
\label{Section-Res}

In this section, we state and prove all our main results.
First, we prove the consistency of $s_{n}^{S}(\mathbf{r}(\hat{\boldsymbol{\beta}}_{S}))$. The proof of the following lemma can be found in the Appendix.
\begin{lemma}
\label{lemma-conv-sca} Assume [R0], [F0] and [X0] hold and that $p/n \rightarrow 0$. Assume also that $f_0$ is strictly decreasing on the non negative real numbers. Then, $s_{n}^{S}(\mathbf{r}(\hat{\boldsymbol{\beta}}_{S})) \overset{P}{%
\rightarrow} s(F_0)$, where $s(F_0)$ is the positive solution of $\mathbb{E} \rho_0(u/s)=b$.
\end{lemma}

It is worth noting that in Theorem 3 of \citep{Davies}, the author proves the consistency of regression S-estimators and the corresponding scale estimates assuming $(p\log n)/n \rightarrow 0$ ($(k_n\log n)/n \rightarrow 0$ in his notation). By using a sharper maximal inequality, this condition can be weakened to $p/n \rightarrow 0$. To do so, simply replace any appeals in the proof of Theorem 3 of \citep{Davies} to the author's Lemma 2 by appeals to our Lemma \ref{lemma2-unif-davies}.

The following lemma is similar to Lemma 1 of \citep{Davies}. 
For $v, s \in\mathbb{R}$, let $
R(v,s)=\mathbb{E} \rho_1 \left( (u-v)/s \right)$.
\begin{lemma}
\label{lemma1-davies}
Assume [R1] and [F0] hold. Then
\begin{enumerate}
\item[(i)] $R: \mathbb{R}\times\mathbb{R}_{+}\rightarrow [0,1]$ is continuous.
\item[(ii)] $R(0,s)\leq R(v,s)$ for $v\in \mathbb{R}$, $s>0$.
\item[(iii)] $R(0,s) < \inf_{\vert v \vert \geq \eta} R(v,s)$ for all $\eta>0$ and $s>0$.
\end{enumerate}
\end{lemma}
\begin{proof}

(i) follows from the fact that $\rho_1$ is bounded and the Bounded Convergence Theorem. 

Next we prove (ii). This is roughly Lemma 3.1 of \citep{MM85}. Note that for any $v\neq 0$, the distribution function $R_v$ of $\vert u - v \vert$ satisfies: $R_v(t) \leq R_{0}(t)$ for all $t>0$ and there exists $\delta>0$ such that $R_v(t)<R_{0}(t)$ for $0 < t \leq \delta$. Since $\rho_{1}(u/s)$ is non decreasing in $\vert u \vert$ and strictly increasing in a neighbourhood of $0$, it follows that for all $s$, $R(v,s)$ has a unique minimum at $v=0$.

Now we prove (iii). Suppose for some $\eta, s>0$, $R(0,s) \geq \inf_{\vert v \vert \geq \eta} R(v,s)$. Note that by [R1] and [F0], $R(0,s)<1$. Take $v_n$ with $\vert v_n\vert \geq \eta$ such that $R(v_n,s)\rightarrow \inf_{\vert v \vert \geq \eta} R(v,s)$. Note that if for some subsequence $v_{n_k}$, $\vert v_{n_k}\vert \rightarrow \infty$, then by the Bounded Convergence Theorem $R(v_{n_k},s)\rightarrow 1$ and hence $R(0,s)\geq 1$, leading to a contradiction. Hence $v_n$ must be bounded. We can assume, eventually passing to a subsequence, that $v_n \rightarrow v^{*}$, with $\vert v^{*} \vert \geq \eta$. Hence $R(v^{*},s) =\inf_{\vert v \vert \geq \eta} R(v,s) \leq R(0,s) $. But by (ii), $R(v,s)$ has a unique minimum at $v=0$. Hence (iii) follows.
\end{proof}

\begin{theorem}[Consistency]
\label{Teo-cons} Assume [R1] and [F0] hold and that $p/n \rightarrow 0$. Then, for any $0<\alpha<1$,
$\eta_n(\alpha) \Vert \hat{\boldsymbol{\beta}} - \boldsymbol{\beta}_0
\Vert \overset{P}{\rightarrow} 0.$
\end{theorem}

\begin{proof}

Fix $0<\alpha <1$. Note that by definition of $\hat{\boldsymbol{\beta }}$%
\begin{equation}
\frac{1}{n}\sum\limits_{i=1}^{n}\rho _{1}\left( \frac{u_{i}-\mathbf{x}%
_{i}^{T}(\hat{\boldsymbol{\beta }}-\boldsymbol{\beta }_{0})}{s_{n}}%
\right) \leq \frac{1}{n}\sum\limits_{i=1}^{n}\rho _{1}\left( \frac{u_{i}}{%
s_{n}}\right).  \notag
\end{equation}

By Lemma \ref{lemma2-unif-davies} in the Appendix, we have that
\begin{equation}
\sup_{\mathbf{b}\in \mathbb{R}^{p},0<s<2s_0}\frac{1}{n}%
\left\vert \sum\limits_{i=1}^{n} \left(\rho _{1}\left( \frac{u_{i}-\mathbf{x}_{i}^{T}\mathbf{b%
}}{s}\right) -R(\mathbf{x}_{i}^{T}\mathbf{b},s) \right)\right\vert \overset{P}{\rightarrow }0.
\label{lemma2-davies}
\end{equation}
Since by assumption $s_n\cp s_0$, Lemma \ref{lemma1-davies} (i) implies that the right
hand side of the last inequality converges in probability to 
\begin{equation}
b^{*} = \mathbb{E} \rho_1 \left( \frac{u}{s_0} \right).  \label{cota_sup}
\end{equation}

By Lemma \ref{lemma1-davies} (ii), $R(0,s)\leq R(v,s)$ for all $v\in \mathbb{R%
}$, $s\in \mathbb{R}$. Then%
\begin{equation}
R(0,s_{n})\leq \frac{1}{n}\sum\limits_{i=1}^{n}R(\mathbf{x}_{i}^{T}(\hat{%
\boldsymbol{\beta }}-\boldsymbol{\beta }_{0}),s_{n}).  \label{minR}
\end{equation}

By Lemma \ref{lemma1-davies} (i)
\begin{equation}
R(0,s_n) \overset{P}{\rightarrow} b^{*}.  \label{cota_inf}
\end{equation}

Then, it follows from \eqref{lemma2-davies}, \eqref{cota_sup}, \eqref{minR} and \eqref{cota_inf} that%
\begin{equation}
\frac{1}{n}\sum\limits_{i=1}^{n}\rho _{1}\left( \frac{u_{i}-\mathbf{x}%
_{i}^{T}(\hat{\boldsymbol{\beta }}-\boldsymbol{\beta }_{0})}{s_{n}}%
\right) \overset{P}{\rightarrow }b^{\ast }  \notag
\end{equation}
and%
\begin{equation}
\frac{1}{n}\sum\limits_{i=1}^{n}R(\mathbf{x}_{i}^{T}(\hat{\boldsymbol{\beta }%
}-\boldsymbol{\beta }_{0}),s_{n})\overset{P}{\rightarrow }b^{\ast }.
\label{conv-R}
\end{equation}

By \eqref{conv-R}, given $\delta >0$, with arbitrarily high probability, for
large enough $n$ we have that%
\begin{equation}
\frac{1}{n}\sum\limits_{i=1}^{n}R(\mathbf{x}_{i}^{T}(\hat{\boldsymbol{\beta }%
}-\boldsymbol{\beta }_{0}),s_{n})\leq b^{\ast }+\delta .
\label{b*+delta}
\end{equation}

Let $\varepsilon>0$, we will show that with arbitrarily high probability, for
large enough $n$, $\eta _{n}(\alpha )\Vert \hat{\boldsymbol{\beta }}-%
\boldsymbol{\beta }_{0}\Vert \leq \varepsilon $. Let $A=\{i:|\mathbf{x}_{i}^{T}(%
\hat{\boldsymbol{\beta }}-\boldsymbol{\beta }_{0})|\geq \varepsilon \}$ and $%
N=\#A$. Then%
\begin{align*}
\frac{1}{n}\sum\limits_{i=1}^{n}R(\mathbf{x}_{i}^{T}(\hat{\boldsymbol{\beta }%
}-\boldsymbol{\beta }_{0}),s_{n})& =\frac{1}{n}\sum\limits_{i\in A}R(%
\mathbf{x}_{i}^{T}(\hat{\boldsymbol{\beta }}-\boldsymbol{\beta }%
_{0}),s_{n})  +\frac{1}{n}\sum\limits_{i\in A^{c}}R(\mathbf{x}_{i}^{T}(\hat{\boldsymbol{%
\beta }}-\boldsymbol{\beta }_{0}),s_{n}).
\end{align*}
Note that%
\begin{equation}
\frac{1}{n}\sum\limits_{i\in A^{c}}R(\mathbf{x}_{i}^{T}(\hat{\boldsymbol{%
\beta }}-\boldsymbol{\beta }_{0}),s_{n})\geq \frac{n-N}{n}R(0,s_{n}). 
\label{cota-Ac}
\end{equation}
Also, if $|\mathbf{x}_{i}^{T}(\hat{%
\boldsymbol{\beta }}-\boldsymbol{\beta }_{0})|\geq \varepsilon $ then $R(\mathbf{x}_{i}^{T}(\hat{\boldsymbol{\beta }}-\boldsymbol{\beta }%
_{0}),s_{n})\geq \inf_{\vert v \vert \geq \varepsilon} R(v,s_n)$. Hence 
\begin{equation}
R(\mathbf{x}_{i}^{T}(\hat{\boldsymbol{\beta }}-\boldsymbol{\beta }_{0}),s_{n}) \geq R(0,s_{n}) +  \left(\inf_{\vert v \vert \geq \varepsilon} R(v,s_n) -R(0,s_{n}) \right).
\nonumber
\end{equation}
We will show that with arbitrarily high probability, for large enough $n$ and $i\in A$
\begin{equation}
R(\mathbf{x}_{i}^{T}(\hat{\boldsymbol{\beta }}-\boldsymbol{\beta }%
_{0}),s_{n})\geq R(0,s_{n})+\kappa,
\label{Eq:conv-inf}
\end{equation}
for some $\kappa =\kappa (\varepsilon )>0$.

First, we will show that
\begin{align}
\sup_{v} \vert R(v,s_n) - R(v,s_0)\vert \cp 0
\label{Eq:Conv-supR}
\end{align}
Fix $u, v \in \mathbb{R}$. Let $\phi_1(t) = \psi_1(t)t$. By [R1], $\phi_1$ is bounded. Applying the Mean Value Theorem we get that, for some $s_{n}^{*}$ such that $\vert s_{n}^{*} - s_0 \vert \leq \vert s_n - s_0\vert$
\begin{align}
\left\vert \rho_{1} \left( \frac{u-v}{s_n}\right) - \rho_{1} \left( \frac{u-v}{s_0}\right) \right\vert & \leq \left\vert \psi_{1} \left( \frac{u-v}{s_{n}^{*}}\right)\left( \frac{u-v}{s_{n}^{*}}\right) \right\vert \left\vert \frac{s_n - s_0}{s_{n}^{*}}\right\vert
\nonumber
\\ & \leq \Vert \phi_1 \Vert_{\infty}  \left\vert \frac{s_n - s_0}{s_{n}^{*}}\right\vert.
\label{Eq:Cota-dif-rho}
\end{align}
Fix some $\eta>0$. Since $s_n \cp s_0$, with arbitrarily high probability, for large enough $n$, the right hand side of \eqref{Eq:Cota-dif-rho} is smaller than $\eta$ for all $u, v$. \eqref{Eq:Conv-supR} is proven.

By Lemma \ref{lemma1-davies} (iii), $\inf_{\vert v \vert \geq \varepsilon} R(v,s_0) > R(0,s_0)$. Let $\eta_1 = ( \inf_{\vert v \vert \geq \varepsilon} R(v,s_0) - R(0,s_0))/4$.
Fix $\eta_2 >0$. Take $n_0$ such that for all $n\geq n_0$, $\sup_{v} \vert R(v,s_n) - R(v,s_0)\vert<\eta_{1}/2$ with probability greater than $1-\eta_2$. For each $n_1 \geq n_0$, take $v_{n_1}$ with $\vert v_{n_1} \vert \geq \varepsilon$ such that $
\inf_{\vert v \vert \geq \varepsilon} R(v,s_{n_1}) \geq R(v_{n_1},s_{n_1}) - \eta_{1}/2$.
Note that $v_{n_1}$ is random.
It follows that with probability greater than $1-\eta_2$, for all $n_1 \geq n_0$
\begin{align*}
\inf_{\vert v \vert \geq \varepsilon} R(v,s_0) - \inf_{\vert v \vert \geq \varepsilon} R(v,s_{n_1})
 &\leq R(v_{n_1},s_0) - R(v_{n_1},s_{n_1}) + \eta_{1}/2
\\ &\leq \sup_{v} \vert R(v,s_{n_1}) - R(v,s_0)\vert +\eta_{1}/2 < \eta_1.
\end{align*}
Since $R(0,s_n) \cp R(0,s_0)$, with arbitrarily high probability, for large enough $n$
\begin{align*}
\inf_{\vert v \vert \geq \varepsilon} R(v,s_n) -R(0,s_{n}) 
\\= \inf_{\vert v \vert \geq \varepsilon} R(v,s_n) - \inf_{\vert v \vert \geq \varepsilon} R(v,s_0) + \inf_{\vert v \vert \geq \varepsilon} R(v,s_0) - R(0,s_{0}) + R(0,s_0) -R(0,s_{n})
\\ \geq 2\eta_1. 
\end{align*}
We have proven \eqref{Eq:conv-inf} for $\kappa(\varepsilon)= \left( \inf_{\vert v \vert \geq \varepsilon} R(v,s_0) - R(0,s_0) \right)/2$.
Hence with arbitrarily high probability, for large enough $n$
\begin{equation}
\frac{1}{n}\sum\limits_{i\in A}R(\mathbf{x}_{i}^{T}(\hat{\boldsymbol{\beta }}%
-\boldsymbol{\beta }_{0}),s_{n})\geq \frac{N}{n}(R(0,s_{n})+\kappa )
\nonumber
\end{equation}%
and thus by \eqref{b*+delta}, \eqref{cota-Ac} and \eqref{Eq:conv-inf} with arbitrarily high probability, for large $n$, we have that if $N\geq (1-\alpha)n$ then $R(0,s_{n})\leq b^{\ast }+\delta -(1-\alpha )\kappa$.
In summary, we have shown that
\begin{equation}
\left\lbrace N\geq (1-\alpha )n \right\rbrace \subseteq \left\lbrace R(0,s_{n})\leq b^{\ast }+\delta -(1-\alpha
)\kappa \right\rbrace \cup A_{n},  \label{contencion-N}
\end{equation}
where $\mathbb{P}(A_{n})\rightarrow 0.$ For any given $\varepsilon $, we can find a
sufficiently small $\delta $ such that $\delta -(1-\alpha )\kappa <0$. Then
by \eqref{cota_inf} and \eqref{contencion-N}, $\mathbb{P}(N\geq (1-\alpha )n)\rightarrow 0$. Hence,
with arbitrarily high probability, for sufficiently large $n$, $n \alpha< n-N$. In this case, there must exist $\mathcal{A}\subset \{1,...,n\}$ with $\#\mathcal{A}=[n\alpha]$ such that $\left\vert \mathbf{x}_{i}^{T}(\hat{\boldsymbol{\beta }}-%
\boldsymbol{\beta }_{0}) \right\vert < \varepsilon$ for all $i \in \mathcal{A}$ and this implies that
\begin{align*}
\eta _{n}(\alpha )\Vert \hat{\boldsymbol{\beta }}-\boldsymbol{\beta }%
_{0}\Vert \leq \min_{\mathcal{A}\subset \{1,...,n\},\#\mathcal{A}=[n\alpha
]}\max_{i\in \mathcal{A}} \left\vert \mathbf{x}_{i}^{T}(\hat{\boldsymbol{\beta }}-%
\boldsymbol{\beta }_{0}) \right\vert \leq \varepsilon,
\end{align*}
which is what we wanted to prove.
\end{proof}

Note that Theorem \ref{Teo-cons} together with [X3] entails that $\hbbet$ is consistent. In the following theorem, we derive its rate of convergence.
Let
\begin{align*}
H_{i}^{n}(\mathbf{x}_{i}^{T} \bbet) = \inf\left\lbrace \psi_{1}^{'}\left( \frac{u_i-v}{s_n} \right) : \vert v \vert \leq \vert \mathbf{x}_{i}^{T} \bbet \vert \right\rbrace.
\end{align*}

\begin{theorem}[Rate of convergence]
\label{Theo-rate}
Assume [R2], [F0] and [X1]-[X5] hold. Assume $(p \log n)/n \rightarrow 0$. Then 
 $\Vert \hat{\bbet} - \bbet_0 \Vert = O_{P}(\sqrt{p/n})$
\end{theorem}
\begin{proof}
A first order Taylor expansion shows that, for some $0\leq \zeta_i \leq 1$,
\begin{align*}
&0\geq \frac{1}{n}(L_{n}(\hbbet)-L_{n}(\boldsymbol{\beta}_{0})) = -\frac{1}{ns_{n}}%
\sum\limits_{i=1}^{n}\psi_{1}\left( \frac{u_{i}}{s_{n}}\right) \mathbf{x}_{i}^{T}(\hbbet-\boldsymbol{\beta}_{0})
\\
&+  \frac{1}{2}\frac{1}{s_{n}^{2}}  \frac{1}{n}\sum\limits_{i=1}^{n}\psi_{1}^{\prime}
\left( \frac {u_{i}-\zeta_{i}(\hbbet-\boldsymbol{\beta}_{0})^{\text{T}}%
\mathbf{x}_{i}}{s_{n}}\right) (\mathbf{x}_{i}^{\text{T}} (\hbbet-\boldsymbol{\beta}_{0}))^{2} 
 \\
&= A_n + B_n.
\end{align*}
Since by assumption $s_n \cp s_0$, by Lemma \ref{lemma-orden1} in the Appendix we have that
\begin{equation}
 A_n = \Vert \hbbet-\boldsymbol{\beta}_{0} \Vert O_{P}\left( \sqrt{\frac{p}{n}} \right).
\label{Cota-An}
\end{equation}
Let $\delta$ and $a^{*}$ be as in Lemma \ref{lemma-mod-portnoy} in the Appendix. Since $\Vert \hbbet-\boldsymbol{\beta}_{0} \Vert \cp 0 $, for sufficiently large $n$, with arbitrarily high probability we have that $\Vert \hbbet-\boldsymbol{\beta}_{0} \Vert <\delta $ and by Lemma \ref{lemma-mod-portnoy}
\begin{align}
B_n &\geq \frac{1}{2\: s_{n}^{2}} \frac{1}{n} \sum\limits_{i=1}^{n} (\mathbf{x}_{i}^{T}(\hbbet-\boldsymbol{\beta}_{0}))^{2} \inf \left\lbrace \psi'_{1} \left( \frac{u_i-v}{s_n}\right) : \vert v \vert \leq \vert \mathbf{x}_{i}^{T}(\hbbet-\boldsymbol{\beta}_{0}) \vert \right\rbrace \nonumber \\ &= \frac{1}{2\: s_{n}^{2}} \frac{1}{n} \sum\limits_{i=1}^{n} (\mathbf{x}_{i}^{T}(\hbbet-\boldsymbol{\beta}_{0}))^{2} H_{i}^{n}(\mathbf{x}_{i}^{T}(\hbbet-\boldsymbol{\beta}_{0})) \nonumber	\\ &\geq \frac{1}{2\: s_{n}^{2}}\Vert \hbbet-\boldsymbol{\beta}_{0} \Vert^{2} \inf_{\Vert \mathbf{z} \Vert =1 ,\Vert  \boldsymbol{\beta} \Vert \leq \delta }\frac{1}{n} \sum\limits_{i=1}^{n} (\mathbf{x}_{i}^{T} \mathbf{z})^{2} H_{i}^{n} (\mathbf{x}_{i}^{T}\boldsymbol{\beta}) \nonumber \\ &\geq \frac{a^{*}}{2\: s_{n}^{2}}\Vert \hbbet-\boldsymbol{\beta}_{0} \Vert^{2}\label{Cota-Bn}.
\end{align}

Hence, it follows from \eqref{Cota-An} and \eqref{Cota-Bn}  that with arbitrarily high probability, for large enough $n$ and some positive constant $M_1$, $
0 \geq A_n+B_n  \geq -M_1\sqrt{p/n} \Vert \hbbet-\boldsymbol{\beta}_{0}) \Vert  + r_n  \Vert 
\hbbet - \boldsymbol{\beta}_0\Vert^{2}$,
where $r_n\cp r_0>0$. Then, $
  \Vert \hbbet - \boldsymbol{\beta}_0\Vert \leq (M_1/r_n) \sqrt{p/n}$,
which proves the theorem.
\end{proof}

Note that under the assumptions of Theorem \ref{Theo-rate}, if we further assume that $\max_{i \leq n} \Vert \mathbf{x}_i \Vert^2 = o(n/p)$ it follows that $\max_{i\leq n} \vert \mathbf{x}_{i}^{T}\left( \hbbet - \bbet_0 \right)\vert \cp 0$. Hence, we can apply Theorem 2 of \citep{Mammen88} to obtain asymptotic expansions for S-estimators.

Next, we derive the asymptotic distribution of $\hat{\bbet}$ .

\begin{theorem}
\label{Theo-Orac}
Assume [R2], [F0] and [X1]-[X6] hold. Assume $(p \log n)/ n \rightarrow 0 $.  Let $\mathbf{a}_n$ be a vector in $\mathbb{R}^{p}$ satisfying $\Vert \mathbf{a}_n \Vert = 1$. Let $r_{n}^{2}= \mathbf{a}_n^{T} \boldsymbol{\Sigma}_{n}^{-1} \mathbf{a}_n$. Then
\begin{equation}
\sqrt{n} r_{n}^{-1} \mathbf{a}_{n}^{T} \left( \hbbet - \bbet_{0} \right) \cw N\left(0,s_0^{2} \frac{a(\psi_1)}{b(\psi_1)^{2} } \right),
\nonumber
\end{equation}
where $a(\psi_1) = \mathbb{E} \psi_{1}^{2}\left(u/s_0 \right)$ and $b(\psi_1) = \mathbb{E} \psi_{1}^{\prime}\left(u/s_0 \right)$.
\end{theorem}

\begin{proof}

From the definition of $\hbbet$, \eqref{Eq:M2}, it follows that
\begin{align*}
\mathbf{0}_{p} &=\frac{1}{\sqrt{n}}\frac{-1}{s_n}\sum\limits_{i=1}^{n}\psi_{1}\left( \frac{y_{i}-\mathbf{x}_{i}^{T}\hat{\bbet}}{s_{n}}\right)\mathbf{x}_{i}.
\end{align*}
Then the Mean Value Theorem gives
\begin{align*}
\mathbf{0}_{p}=\frac{1}{\sqrt{n}}\frac{-1}{s_n}\sum\limits_{i=1}^{n}\psi_{1}\left( \frac{u_i}{s_{n}}\right)\mathbf{x}_{i}&+
\frac{1}{s_n^{2}}\mathbf{W}_{n}\sqrt{n}(\hat{\bbet}-\bbet_{0}),
\end{align*}
where 
\begin{equation}
\mathbf{W}_{n} = \frac{1}{n}\sum\limits_{i=1}^{n}\psi_{1}'\left( \frac{u_{i}-\zeta_i \mathbf{x}_{i}^{T} (\hbbet - \bbet_{0}) }{s_{n}}\right)\mathbf{x}_{i}\mathbf{x}_{i}^{T}
\nonumber
\end{equation}
and $0\leq \zeta_i \leq 1$. Let 
\begin{align*}
\mathbf{W}_{n}^{1} &= \frac{1}{n}\sum\limits_{i=1}^{n}\psi_{1}'\left( \frac{u_i }{s_{n}}\right)\mathbf{x}_{i}\mathbf{x}_{i}^{T}, \:\:
\mathbf{W}_{n}^{2} = \mathbb{E} \psi_{1}^{\prime} \left( \frac{u}{s_n}\right) \boldsymbol{\Sigma}_{n},
\nonumber
\end{align*}
where the expectation in $\mathbb{E} \psi_{1}^{\prime} \left( u/s_n\right)$ is taken only with respect to $u$.
Then
\begin{align*}
\sqrt{n} \mathbf{a}_{n}^{T}(\hat{\bbet}-\bbet) &= \frac{s_{n}}{\mathbb{E} \psi_{1}^{\prime} \left( u/s_n\right)} \frac{1}{\sqrt{n}} \sum\limits_{i=1}^{n}\psi_{1}\left( \frac{u_i}{s_{n}}\right)\mathbf{a}_{n}^{T} \boldsymbol{\Sigma}_{n}^{-1} \mathbf{x}_{i}
\\&- \frac{1}{\mathbb{E} \psi_{1}^{\prime} \left( u/s_n\right)}\mathbf{a}_{n}^{T}\boldsymbol{\Sigma}_{n}^{-1}(\mathbf{W}_{n}- \mathbf{W}_{n}^{1})\sqrt{n}(\hat{\bbet}-\bbet_{0})
\\&- \frac{1}{\mathbb{E} \psi_{1}^{\prime} \left( u/s_n\right)}\mathbf{a}_{n}^{T}\boldsymbol{\Sigma}_{n}^{-1}(\mathbf{W}_{n}^{1}- \mathbf{W}_{n}^{2})\sqrt{n}(\hat{\bbet}-\bbet_{0})
\\&= \frac{s_{n}}{\mathbb{E} \psi_{1}^{\prime} \left( u/s_n\right)} \frac{1}{\sqrt{n}} \sum\limits_{i=1}^{n}\psi_{1}\left( \frac{u_i}{s_{n}}\right)\mathbf{a}_{n}^{T} \boldsymbol{\Sigma}_{n}^{-1} \mathbf{x}_{i}  
\\&+ A_n +B_n.
\end{align*}
We will show that $A_n +B_n= o_{P}(1)$. Note that by [R2] and the Bounded Convergence Theorem, $\mathbb{E} \psi_{1}^{\prime} \left( u/s_n\right) \cp \mathbb{E} \psi_{1}^{\prime} \left( u/s_0\right)$.

For a matrix $\mathbf{W}$ let $\Vert \mathbf{W} \Vert$ be its spectral norm and let $\Vert \mathbf{W} \Vert_{F}$ be its Frobenius norm. Recall that for any $\mathbf{W}$, $\Vert\mathbf{W}  \Vert \leq \Vert \mathbf{W}\Vert_{F}$. We will show that $\Vert \mathbf{W}_n - \mathbf{W}_{n}^{1} \Vert = o_{P}(1/ \sqrt{p})$ and $\Vert \mathbf{W}_{n}^{1} - \mathbf{W}_{n}^{2} \Vert=o_{P}(1/ \sqrt{p})$. Take $\mathbf{b}\in\mathbb{R}^{p}$ with $\Vert \mathbf{b} \Vert=1$. Then, applying the Mean Value Theorem, we get
\begin{align*}
&\vert \mathbf{b}^{T} (\mathbf{W}_{n} - \mathbf{W}_{n}^{1}) \mathbf{b} \vert \\ &\leq  \frac{1}{n}\sum\limits_{i=1}^{n} \left\vert \psi_{1}'\left( \frac{u_i }{s_{n}}\right)- \psi_{1}'\left( \frac{u_{i}-\zeta_i \mathbf{x}_{i}^{T} (\hbbet - \bbet_{0}) }{s_{n}}\right) \right \vert (\mathbf{b}^{T}\mathbf{x}_{i})^{2}  \\ &\leq \frac{1}{n} \sum\limits_{i=1}^{n} \frac{\Vert \psi_{1}^{\prime\prime}\Vert_{\infty}}{s_{n}} \vert \mathbf{x}_{i}^{T}(\hbbet - \bbet_{0}) \vert (\mathbf{b}^{T}\mathbf{x}_{i})^{2} \leq \frac{\Vert \psi_{1}^{\prime\prime}\Vert_{\infty}}{s_{n}} \max_{i\leq n} \Vert \mathbf{x}_{i} \Vert \Vert \hbbet - \bbet_{0}\Vert \tau.
\end{align*}
Since $\Vert \hbbet - \bbet_{0}\Vert=O_{P}(\sqrt{p/n})$, taking supremum over $\mathbf{b}$, from [X6] it follows that
\begin{align*}
\Vert \mathbf{W}_n - \mathbf{W}_{n}^{1}\Vert = o_{P}\left(\frac{1}{\sqrt{p}}\right)
\end{align*}
and hence we have that $A_n = o_{P}(1)$.
By Lemma \ref{lemma-matriz} in the Appendix and [X6],
\begin{align*}
\Vert \mathbf{W}_{n}^{1} - \mathbf{W}_{n}^{2} \Vert_{F} = O_{P}\left( \sqrt{p/n} \max_{i \leq n} \Vert \mathbf{x}_{i}\Vert \right) = o_{P}\left(\frac{1}{\sqrt{p}}\right)
\end{align*}
and hence we have that $B_n = o_{P}(1)$.

We have thus shown that $A_n+B_n=o_{P}(1)$ and so it follows that
\begin{align*}
\sqrt{n} \mathbf{a}_{n}^{T}(\hat{\bbet}-\bbet_{0}) &= \frac{s_{n}}{\mathbb{E} \psi_{1}^{\prime} \left( u/s_n\right)} \frac{1}{\sqrt{n}} \sum\limits_{i=1}^{n}\psi_{1}\left( \frac{u_i}{s_{n}}\right)\mathbf{a}_{n}^{T} \boldsymbol{\Sigma}_{n}^{-1} \mathbf{x}_{i}+o_{P}(1).
\end{align*}
Note that $r_n$ and $1/r_n$ are bounded. By Lemma \ref{lemma-tight} in the Appendix
\begin{equation}
r_{n}^{-1}\frac{1}{\sqrt{n}} \sum\limits_{i=1}^{n}\psi_{1}\left( \frac{u_i}{s_{n}}\right)\mathbf{a}_{n}^{T} \boldsymbol{\Sigma}_{n}^{-1} \mathbf{x}_{i} \cw N(0,a(\psi_1)).
\nonumber
\end{equation}
The theorem now follows from Slutzky's Theorem.
\end{proof}

\section{Appendix}
\label{Section-App}
\begin{proof}[Proof of Lemma \ref{lemma:cond_davies}]

Let $0<\alpha<1$ be such that $\liminf \eta_n(\alpha)>0$. Note that for all $\eta>0$, $\#\lbrace i: \Vert \mathbf{x}_i \Vert \geq \eta \rbrace \leq n M p \eta^{-2}$. Take $\eta=\sqrt{2Mp(1-\alpha)^{-1}}$. Let $\alpha_1= (1/2) (1+\alpha)$, $\eta_1 = \sqrt{2M(1-\alpha)^{-1}}$ and $\mathcal{A}= \lbrace i: \Vert \mathbf{x}_i \Vert < \eta_1 \sqrt{p} \rbrace$. Then $\#\mathcal{A} \geq n \alpha_1$, with $0<\alpha<\alpha_1 < 1$.

Take $\boldsymbol{\theta}^{*}$ with $\Vert \boldsymbol{\theta}^{*} \Vert =1$ such that
\begin{align*}
\sum\limits_{i \in \mathcal{A}} \vert \mathbf{x}_{i}^{T} \boldsymbol{\theta}^{*} \vert^{2} = \min_{\Vert \boldsymbol{\theta} \Vert =1} \sum\limits_{i \in \mathcal{A}} \vert \mathbf{x}_{i}^{T} \boldsymbol{\theta}\vert^{2}.
\end{align*}
Let $\mathcal{G}$ be the set of $i\in \mathcal{A}$ giving rise to the smallest $[n \alpha]$ values of $\vert \mathbf{x}_{i}^{T} \boldsymbol{\theta}^{*} \vert$. Then, by definition of $\eta_{n}(\alpha)$, $\eta_{n}(\alpha)\leq \max_{i\in\mathcal{G}} \vert \mathbf{x}_{i}^{T} \boldsymbol{\theta}^{*} \vert$. Hence, $ \eta_n(\alpha) \leq  \vert \mathbf{x}_{i}^{T} \boldsymbol{\theta}^{*} \vert$ for all $i\in \mathcal{A}\setminus\mathcal{G}$.
Thus
\begin{align*}
\min_{\Vert \boldsymbol{\theta} \Vert =1} \frac{1}{n}\sum\limits_{i \in \mathcal{A}} \vert \mathbf{x}_{i}^{T} \boldsymbol{\theta}\vert^{2} =  \frac{1}{n} \sum\limits_{i \in \mathcal{A}} \vert \mathbf{x}_{i}^{T} \boldsymbol{\theta}^{*} \vert^{2} &\geq \frac{1}{n} \sum\limits_{i \in \mathcal{A}\setminus\mathcal{G}} \vert \mathbf{x}_{i}^{T} \boldsymbol{\theta}^{*} \vert^{2}
\\ &\geq \frac{(n\alpha_1 - [n\alpha])\eta_{n}(\alpha)^{2}}{n} 
\\ &\geq (\alpha_1 - \alpha) \eta_n(\alpha)^{2}.
\end{align*}
The lemma is proven.
\end{proof}

We will make extensive use of the tools from empirical processes theory that appear in \citep{Pollard89} and \citep{vdv}. The results in \citep{Pollard89}, in particular the maximal inequalities of Theorem 4.2, are stated for i.i.d random variables. In Theorem \ref{Theo-Pollard} we adapt Theorem 4.2 of \citep{Pollard89} to make it directly applicable to our scenario of interest.

We first introduce some notation. Let $\varepsilon>0$. Let $\mathcal{H}$ be a class of funcions defined on $\mathbb{R}^{d}$ and let $\Vert . \Vert$ be a pseudo-norm on $\mathcal{H}$.
\begin{itemize}
\item The capacity number of $\mathcal{H}$, $D(\varepsilon,\mathcal{H},\Vert . \Vert)$, is the largest $N$ such that there exists $h_1,\dots,h_N$ in $\mathcal{H}$ with $\Vert h_i - h_j\Vert >\varepsilon$ for all $i\neq j$. The capacity number is also called the packing number in the literature.
\item The covering number of $\mathcal{H}$, $N(\varepsilon,\mathcal{H},\Vert . \Vert)$, is the minimal number of open balls of radius $\varepsilon$ needed to cover $\mathcal{H}$.
\item Given two functions $h, g$ a bracket $[h,g]$ is the set of all functions $f$ such that $h\leq f \leq g$. An $\varepsilon$-bracket is a bracket $[h,g]$ such that $\Vert h-g \Vert<\varepsilon$. $N_{[\:]}(\varepsilon,\mathcal{H},\Vert . \Vert)$ is the bracketing number of $\mathcal{H}$, that is, the minimum number of $\varepsilon$-brackets needed to cover $\mathcal{H}$.
\item Given a metric space $(T,d)$, the covering number of $T$, $N(\varepsilon,T,d)$, is the minimal number of open balls of radius $\varepsilon$ needed to cover $T$.
\end{itemize}
It is easy to show that $D(\varepsilon,\mathcal{H},\Vert . \Vert) \leq N(\varepsilon/2,\mathcal{H},\Vert . \Vert) \leq N_{[\:]}(\varepsilon,\mathcal{H},\Vert . \Vert)$. Given $Q$, a probability measure on $\mathbb{R}^{d}$ with finite support, let $\Vert . \Vert_{2,Q}$ be the $L^{2}(Q)$ pseudo-norm.

\begin{theorem}
\label{Theo-Pollard}
Let $\mathbf{z}_{1},\dots,\mathbf{z}_{n}$ be fixed vectors in $\mathbb{R}^{d}$. Let $\mathbf{v}_1,\dots,\mathbf{v}_n$ be i.i.d. random vectors in $\mathbb{R}^{m}$. Let $\mathcal{H}$ be a class of functions defined in $\mathbb{R}^{m+d}$ and taking values in $\mathbb{R}$. Assume $\mathcal{H}$ has envelope $H$ that satisfies
\begin{align*}
\frac{1}{n} \sum\limits_{i=1}^{n} \mathbb{E}H^{2}(\mathbf{v}_{i},\mathbf{z}_{i})<\infty
\end{align*} and that $\mathcal{H}$ contains the zero function. Furthermore, assume that there exists a decreasing function $D(\varepsilon)$ that satisfies $\int_{0}^{1} \left( \log D(\varepsilon)\right)^{1/2} d\varepsilon < \infty$, such that for all $0<\varepsilon<1$ and any probability measure on $\mathbb{R}^{m+d}$ with finite support $Q$ with $\Vert H \Vert_{2,Q}>0$, $
D(\varepsilon \Vert H\Vert_{2,Q}, \mathcal{H},\Vert . \Vert_{2,Q})\leq D(\varepsilon)$.
Then
\begin{enumerate}
\item[(i)]\begin{align*}
\mathbb{E}\sup_{\mathcal{H}}\left\vert \frac{1}{\sqrt{n}} \sum\limits_{i=1}^{n} \left( h(\mathbf{v}_i,\mathbf{z}_i) - \mathbb{E}h(\mathbf{v}_i,\mathbf{z}_i)\right)\right\vert 
\\ \leq M \left(\frac{1}{n}\sum\limits_{i=1}^{n} \mathbb{E}H^{2}(\mathbf{v}_{i},\mathbf{z}_{i}) \right)^{1/2}  \left( \int\limits_{0}^{1} \left( \log D(\varepsilon) \right)^{1/2} d\varepsilon \right)       ,
\end{align*}
\item[(ii)]\begin{align*}
\mathbb{E}\sup_{\mathcal{H}}\left\vert \frac{1}{\sqrt{n}} \sum\limits_{i=1}^{n} \left( h(\mathbf{v}_i,\mathbf{z}_i) - \mathbb{E}h(\mathbf{v}_i,\mathbf{z}_i)\right)\right\vert^{2} 
\\ \leq M \frac{1}{n}\sum\limits_{i=1}^{n}\mathbb{E}H^{2}(\mathbf{v}_{i},\mathbf{z}_{i}) \: \left( \int\limits_{0}^{1} \left( \log D(\varepsilon) \right)^{1/2} d\varepsilon \right)^{2} , 
\end{align*}
\end{enumerate}
where $M>0$ is a fixed universal constant.
\end{theorem}
\begin{proof}
The proof of this theorem is a simple adaptation of the proof of Theorem 4.2 of \citep{Pollard89}.

We prove (ii). Let $\mathbb{P}_n$ be the empirical probability measure that places mass $1/n$ at each of the points $(\mathbf{v}_{i},\mathbf{z}_{i})$ $i=1,\dots,n$. Let $\Vert . \Vert_{2,n}$ be the $L^{2}(\mathbb{P}_n)$ pseudo-norm.
Let $\tilde{\mathbf{v}}_1,\dots,\tilde{\mathbf{v}}_n$ be i.i.d. random vectors independent of and with the same distribution as $\mathbf{v}_1,\dots,\mathbf{v}_n$. With a slight abuse of notation denote $\mathbf{v}=(\mathbf{v}_1,\dots,\mathbf{v}_n)$ and let $\mathbb{E}_{\mathbf{v}}$ be the expectation conditional on $\mathbf{v}$. It follows that for all $i=1,\dots,n$, $h(\mathbf{v}_i,\mathbf{z}_i)=\mathbb{E}_{\mathbf{v}}h(\mathbf{v}_i,\mathbf{z}_i)$ and $\mathbb{E}h(\mathbf{v}_i,\mathbf{z}_i)=\mathbb{E}h(\tilde{\mathbf{v}}_i,\mathbf{z}_i)=\mathbb{E}_{\mathbf{v}}h(\tilde{\mathbf{v}}_i,\mathbf{z}_i)$. Then, for all $h\in\mathcal{H}$
\begin{align*}
\frac{1}{\sqrt{n}} \sum\limits_{i=1}^{n} \left( h(\mathbf{v}_i,\mathbf{z}_i) - \mathbb{E}h(\mathbf{v}_i,\mathbf{z}_i)\right)=\mathbb{E}_{\mathbf{v}} \frac{1}{\sqrt{n}} \sum\limits_{i=1}^{n} \left( h(\mathbf{v}_i,\mathbf{z}_i) - h(\tilde{\mathbf{v}}_{i},\mathbf{z}_i)\right).
\end{align*}
By Jensen's inequality
\begin{align*}
\left\vert \mathbb{E}_{\mathbf{v}} \frac{1}{\sqrt{n}} \sum\limits_{i=1}^{n} \left( h(\mathbf{v}_i,\mathbf{z}_i) - h(\tilde{\mathbf{v}}_{i},\mathbf{z}_i)\right) \right\vert^{2} \leq \mathbb{E}_{\mathbf{v}} \left\vert  \frac{1}{\sqrt{n}} \sum\limits_{i=1}^{n} \left( h(\mathbf{v}_i,\mathbf{z}_i) - h(\tilde{\mathbf{v}}_{i},\mathbf{z}_i)\right) \right\vert^{2}.
\end{align*}
Hence
\begin{align*}
\mathbb{E}\sup_{\mathcal{H}}\left\vert \frac{1}{\sqrt{n}} \sum\limits_{i=1}^{n} \left( h(\mathbf{v}_i,\mathbf{z}_i) - \mathbb{E}h(\mathbf{v}_i,\mathbf{z}_i)\right)\right\vert^{2} \leq \mathbb{E}\sup_{\mathcal{H}} \mathbb{E}_{\mathbf{v}} \left\vert  \frac{1}{\sqrt{n}} \sum\limits_{i=1}^{n} \left( h(\mathbf{v}_i,\mathbf{z}_i) - h(\tilde{\mathbf{v}}_{i},\mathbf{z}_i)\right) \right\vert^{2}
\\ \leq \mathbb{E} \mathbb{E}_{\mathbf{v}} \sup_{\mathcal{H}} \left\vert  \frac{1}{\sqrt{n}} \sum\limits_{i=1}^{n} \left( h(\mathbf{v}_i,\mathbf{z}_i) - h(\tilde{\mathbf{v}}_{i},\mathbf{z}_i)\right) \right\vert^{2}
=\mathbb{E} \sup_{\mathcal{H}} \frac{1}{n}\left\vert   \sum\limits_{i=1}^{n} \left( h(\mathbf{v}_i,\mathbf{z}_i) - h(\tilde{\mathbf{v}}_{i},\mathbf{z}_i)\right) \right\vert^{2}.
\end{align*}
Let $g_1,\dots,g_n$ be i.i.d random variables, independent of $\tilde{\mathbf{v}}_1,\dots,\tilde{\mathbf{v}}_n$ and of $\mathbf{v}_1,\dots,\mathbf{v}_n$ such that $g_i \sim N(0,1)$. Define $\sigma_i=g_i /\vert g_i\vert$ for $i=1,\dots,n$. Then $\sigma_1,\dots,\sigma_n$ are independent of $\tilde{\mathbf{v}}_1,\dots,\tilde{\mathbf{v}}_n$ and of $\mathbf{v}_1,\dots,\mathbf{v}_n$. Note that $\mathbb{P}(\sigma_i=1)=\mathbb{P}(\sigma_i=-1)=1/2$ and that $\sigma_i$ is independent of $\vert g_i\vert$. Let $\boldsymbol{\sigma}=(\sigma_1,\dots,\sigma_n)$. By the symmetry between $\tilde{\mathbf{v}}_i$ and $\mathbf{v}_i$ we have that
\begin{align*}
\mathbb{E} \sup_{\mathcal{H}} \frac{1}{n}\left\vert   \sum\limits_{i=1}^{n} \left( h(\mathbf{v}_i,\mathbf{z}_i) - h(\tilde{\mathbf{v}}_{i},\mathbf{z}_i)\right) \right\vert^{2} = \mathbb{E} \sup_{\mathcal{H}} \frac{1}{n}\left\vert   \sum\limits_{i=1}^{n} \sigma_i \left( h(\mathbf{v}_i,\mathbf{z}_i) - h(\tilde{\mathbf{v}}_{i},\mathbf{z}_i)\right) \right\vert^{2}.
\end{align*}
Now 
\begin{align*}
\sup_{\mathcal{H}} \left\vert   \sum\limits_{i=1}^{n}\sigma_i \left( h(\mathbf{v}_i,\mathbf{z}_i) - h(\tilde{\mathbf{v}}_{i},\mathbf{z}_i)\right) \right\vert\leq \sup_{\mathcal{H}} \left\vert   \sum\limits_{i=1}^{n}\sigma_i h(\mathbf{v}_i,\mathbf{z}_i) \right\vert + \sup_{\mathcal{H}} \left\vert   \sum\limits_{i=1}^{n}\sigma_i h(\tilde{\mathbf{v}}_i,\mathbf{z}_i) \right\vert.
\end{align*}
Hence
\begin{align*}
\mathbb{E} \sup_{\mathcal{H}} \frac{1}{n}\left\vert   \sum\limits_{i=1}^{n} \sigma_i \left( h(\mathbf{v}_i,\mathbf{z}_i) - h(\tilde{\mathbf{v}}_{i},\mathbf{z}_i)\right) \right\vert^{2} \leq 4 \mathbb{E} \sup_{\mathcal{H}} \frac{1}{n}\left\vert   \sum\limits_{i=1}^{n}\sigma_i h(\mathbf{v}_i,\mathbf{z}_i) \right\vert^{2}.
\end{align*}
Let $\gamma$ be the expectation of $\vert g_1 \vert$ and let $\mathbb{E}_{\mathbf{v},\boldsymbol{\sigma}}$ be the expectation conditional on $\mathbf{v}$ and $\boldsymbol{\sigma}$. Then for all $i=1,\dots,n$,
$\mathbb{E}_{\mathbf{v},\boldsymbol{\sigma}} \sigma_i  h(\mathbf{v}_i,\mathbf{z}_i)=\sigma_i  h(\mathbf{v}_i,\mathbf{z}_i)$ and $\mathbb{E}_{\mathbf{v},\boldsymbol{\sigma}} \vert g_i\vert=\gamma$. Hence, applying Jensen's inequality
\begin{align*}
\mathbb{E} \sup_{\mathcal{H}} \frac{1}{n}\left\vert   \sum\limits_{i=1}^{n}\sigma_i h(\mathbf{v}_i,\mathbf{z}_i) \right\vert^{2} &= \mathbb{E} \sup_{\mathcal{H}} \frac{1}{n}\left\vert   \sum\limits_{i=1}^{n}\sigma_i h(\mathbf{v}_i,\mathbf{z}_i) \mathbb{E}_{\mathbf{v},\boldsymbol{\sigma}} \vert g_i\vert / \gamma\right\vert^{2}
\\ = \mathbb{E} \sup_{\mathcal{H}} \frac{1}{n}\left\vert   \mathbb{E}_{\mathbf{v},\boldsymbol{\sigma}}\sum\limits_{i=1}^{n}\sigma_i h(\mathbf{v}_i,\mathbf{z}_i)  \vert g_i\vert / \gamma\right\vert^{2} &= \mathbb{E} \sup_{\mathcal{H}} \frac{1}{n}\left\vert   \mathbb{E}_{\mathbf{v},\boldsymbol{\sigma}}\sum\limits_{i=1}^{n} g_i h(\mathbf{v}_i,\mathbf{z}_i)/\gamma\right\vert^{2}
\\ \leq \mathbb{E} \sup_{\mathcal{H}} \mathbb{E}_{\mathbf{v},\boldsymbol{\sigma}} \frac{1}{n}\left\vert   \sum\limits_{i=1}^{n} g_i h(\mathbf{v}_i,\mathbf{z}_i)/\gamma\right\vert^{2}
& \leq \mathbb{E}  \mathbb{E}_{\mathbf{v},\boldsymbol{\sigma}} \sup_{\mathcal{H}} \frac{1}{n}\left\vert   \sum\limits_{i=1}^{n} g_i h(\mathbf{v}_i,\mathbf{z}_i)/\gamma\right\vert^{2}
\\ &= \gamma^{-2} \mathbb{E} \sup_{\mathcal{H}} \frac{1}{n}\left\vert   \sum\limits_{i=1}^{n} g_i h(\mathbf{v}_i,\mathbf{z}_i)\right\vert^{2}.
\end{align*}
In summary, we have shown that
\begin{align}
\mathbb{E}\sup_{\mathcal{H}}\left\vert \frac{1}{\sqrt{n}} \sum\limits_{i=1}^{n} \left( h(\mathbf{v}_i,\mathbf{z}_i) - \mathbb{E}h(\mathbf{v}_i,\mathbf{z}_i)\right)\right\vert^{2} \leq 4 \gamma^{-2} \mathbb{E} \sup_{\mathcal{H}} \frac{1}{n}\left\vert   \sum\limits_{i=1}^{n} g_i h(\mathbf{v}_i,\mathbf{z}_i)\right\vert^{2}.
\label{Eq:6-Pollard}
\end{align}
Define for $h\in\mathcal{H}$, $
Z_n(h,\mathbf{v})=(1/\sqrt{n})\sum\limits_{i=1}^{n} g_i h(\mathbf{v}_i,\mathbf{z}_i)$.
Then \eqref{Eq:6-Pollard} can be written as
\begin{align}
\mathbb{E}\sup_{\mathcal{H}}\left\vert \frac{1}{\sqrt{n}} \sum\limits_{i=1}^{n} \left( h(\mathbf{v}_i,\mathbf{z}_i) - \mathbb{E}h(\mathbf{v},\mathbf{z}_i)\right)\right\vert^{2} \leq 4 \gamma^{-2}  \mathbb{E} \sup_{\mathcal{H}} \left\vert Z_n(h,\mathbf{v}) \right\vert^{2}.
\label{Eq:7-Pollard}
\end{align}
Note that, conditionally on the $\mathbf{v}_i$, $Z_n$ is a zero-mean Gaussian process with increments bounded by the $L^{2}(\mathbb{P}_n)$ pseudo-norm: for all $h_1,h_2\in\mathcal{H}$
\begin{align*}
&\mathbb{E}_{\mathbf{v}} \left\vert Z_n(h_1,\mathbf{v}) - Z_n(h_2,\mathbf{v}) \right\vert^{2} 
\\ &= \frac{1}{n}\mathbb{E}_{\mathbf{v}} \sum_{i,j} (h_1(\mathbf{v}_i,\mathbf{z}_i) - h_2(\mathbf{v}_i,\mathbf{z}_i)) (h_1(\mathbf{v}_j,\mathbf{z}_j) - h_2(\mathbf{v}_j,\mathbf{z}_j)) g_i g_j 
\\ &= \frac{1}{n} \sum_{i,j} (h_1(\mathbf{v}_i,\mathbf{z}_i) - h_2(\mathbf{v}_i,\mathbf{z}_i)) (h_1(\mathbf{v}_j,\mathbf{z}_j) - h_2(\mathbf{v}_j,\mathbf{z}_j)) \mathbb{E}_{\mathbf{v}}g_i g_j 
\\ &= \frac{1}{n} \sum_{i=1}^{n} (h_1(\mathbf{v}_i,\mathbf{z}_i) - h_2(\mathbf{v}_i,\mathbf{z}_i))^{2} = \Vert h_1 -h_2\Vert_{2,n}^{2}.
\end{align*}
Also, for fixed $\mathbf{v}$, $Z_n$ has continuous sample paths in the $L^{2}(\mathbb{P}_n)$ pseudo-norm: if $\Vert h -h_k\Vert_{2,n} \rightarrow 0$ when $k\rightarrow \infty$ then $h_k(\mathbf{v}_i,\mathbf{z}_i) \rightarrow h(\mathbf{v}_i,\mathbf{z}_i)$ for all $i=1,\dots,n$ and hence $Z_n(h_k,\mathbf{v}) \rightarrow Z_n(h,\mathbf{v})$ for each realization of the $g_i$.
Therefore, we can apply Theorem 3.3 of \citep{Pollard89}: there exists an universal constant $K>0$ such that
\begin{align}
\left( \mathbb{E}_{\mathbf{v}} \sup_{\mathcal{H}} \left\vert Z_n(h,\mathbf{v}) \right\vert^{2} \right)^{1/2} \leq K \int_{0}^{\Delta(\mathbf{v})} \left( \log D(x,\mathcal{H},\Vert . \Vert_{2,n}) \right)^{1/2} dx,
\label{Eq:Teo3-Pollard}
\end{align}
where $\Delta(\mathbf{v})=\sup_{\mathcal{H}} \Vert h \Vert_{2,n}$. 

If $\Vert H \Vert_{2,n}>0$, since by assumption $D(\varepsilon\Vert H \Vert_{2,n},\mathcal{H},\Vert . \Vert_{2,n}) \leq D(\varepsilon)$ for all $0<\varepsilon<1$, we have that
\begin{align*}
 \int_{0}^{1} \left( \log D(\varepsilon\Vert H \Vert_{2,n},\mathcal{H},\Vert . \Vert_{2,n}) \right)^{1/2} d\varepsilon \leq \int_{0}^{1} \left( \log D(\varepsilon) \right)^{1/2} d\varepsilon < \infty.
\end{align*}
Also, since $\Delta(\mathbf{v})/\Vert H \Vert_{2,n} \leq 1$
\begin{align*}
\int_{0}^{\Delta(\mathbf{v})/\Vert H \Vert_{2,n}} \left( \log D(\varepsilon\Vert H \Vert_{2,n},\mathcal{H},\Vert . \Vert_{2,n}) \right)^{1/2} d\varepsilon
 \\ \leq \int_{0}^{1} \left( \log D(\varepsilon\Vert H \Vert_{2,n},\mathcal{H},\Vert . \Vert_{2,n}) \right)^{1/2} d\varepsilon
\end{align*}
The change of variables $x=\varepsilon\Vert H \Vert_{2,n}$ gives
\begin{align*}
\Vert H \Vert_{2,n} \int_{0}^{\Delta(\mathbf{v})/\Vert H \Vert_{2,n}} \left( \log D(\varepsilon\Vert H \Vert_{2,n},\mathcal{H},\Vert . \Vert_{2,n}) \right)^{1/2} d\varepsilon
\\=
\int_{0}^{\Delta(\mathbf{v})} \left( \log D(x,\mathcal{H},\Vert . \Vert_{2,n}) \right)^{1/2} dx. 
\end{align*}
Then
\begin{align}
\left( \mathbb{E}_{\mathbf{v}} \sup_{\mathcal{H}} \left\vert Z_n(h,\mathbf{v}) \right\vert^{2} \right)^{1/2} \leq K \Vert H \Vert_{2,n} \int_{0}^{1} \left( \log D(\varepsilon) \right)^{1/2} d\varepsilon.
\label{Eq:8-Pollard}
\end{align}
On the other hand, if $\Vert H \Vert_{2,n}=0$ then  $\Delta(\mathbf{v})=0$ and this implies that the right hand side of \eqref{Eq:Teo3-Pollard} is zero. In this case \eqref{Eq:8-Pollard} holds trivially. 

We have thus shown that
\begin{align*}
&\mathbb{E}\sup_{\mathcal{H}}\left\vert \frac{1}{\sqrt{n}} \sum\limits_{i=1}^{n} \left( h(\mathbf{v}_i,\mathbf{z}_i) - \mathbb{E}h(\mathbf{v}_i,\mathbf{z}_i)\right)\right\vert^{2} 
\\ & \leq 4 \gamma^{-2}  \mathbb{E} \sup_{\mathcal{H}} \left\vert Z_n(h,\mathbf{v}) \right\vert^{2} 
\\ &= 4 \gamma^{-2}  \mathbb{E}\mathbb{E}_{\mathbf{v}} \sup_{\mathcal{H}} \left\vert Z_n(h,\mathbf{v}) \right\vert^{2}
\\ &\leq 4 \gamma^{-2} K^{2}  \mathbb{E}\Vert H \Vert_{2,n}^{2} \left( \int_{0}^{1}\left( \log D(\varepsilon) \right)^{1/2} d\varepsilon \right)^{2}
\\ &= 4 \gamma^{-2} K^{2}  \frac{1}{n}\sum\limits_{i=1}^{n} \mathbb{E}H^{2}(\mathbf{v}_i,\mathbf{z}_i) \left( \: \int_{0}^{1}\left( \log D(\varepsilon) \right)^{1/2} d\varepsilon\right)^{2} ,
\end{align*}
which is what we wanted to prove.
Part (i) can be proved by substituting $L^{1}(\mathbb{P}_n)$ norms by $L^{2}(\mathbb{P}_n)$ in the arguments leading to \eqref{Eq:7-Pollard} and then applying Theorem 3.2 of \citep{Pollard89}.
\end{proof}

The following lemma is a key result in the proof of the consistency of the estimators.
\begin{lemma}
\label{lemma2-unif-davies}
Assume $\rho$ is a bounded $\rho$-function. Consider the class of functions
\begin{align*}
\mathcal{H}=\left\lbrace h_{s,\mathbf{b}} (u,\mathbf{x}) = \rho \left( \frac{u-\mathbf{x}^{T}\mathbf{b%
}}{s}\right) :  \: \mathbf{b}\in \mathbb{R}^{p} \:, s>0\right\rbrace.
\end{align*}
Then, if $p/n \rightarrow 0$,
\begin{align*}
\sup_{h\in\mathcal{H}} \left\vert (1/n) \sum\limits_{i=1}^{n}  \left( h(u_i,\mathbf{x}_{i}) - \mathbb{E}h(u,\mathbf{x}_{i})   \right) \right\vert \cp 0.
\end{align*}
\end{lemma}
\begin{proof}

We will apply the maximal inequalities of Theorem \ref{Theo-Pollard} to $\mathcal{H}\cup \lbrace 0 \rbrace$.

Let $
\mathcal{L}=\left\lbrace l_{s,\mathbf{b}}(u,\mathbf{x})= (u - \mathbf{x}^{T} \mathbf{b})/s: \mathbf{b}\in \mathbb{R}^{p}, \: s>0 \right\rbrace$.
Then $\mathcal{L}$ is a subset of the vector space of all linear functions in $p+1$ variables. This vector space has dimension $p+1$. It follows from Lemma 2.6.15 of \citep{vdv} that $\mathcal{L}$ has VC-index at most $p+3$.

Note that $\rho = m^1 + m^2$, where $m^1(x)=\rho(x)I\lbrace x\geq 0 \rbrace$ and $m^2(x)=\rho(x)I\lbrace x<0 \rbrace$. Note that $m^1$ is non-decreasing and $m^2$ is non-increasing. By Lemma 9.9 (viii) of \citep{Kosorok}, $m^1 \circ \mathcal{L}$ and $m^2\circ\mathcal{L}$ have VC-index at most $p+3$. $m^1\circ\mathcal{L}$ and $m^2\circ\mathcal{L}$ have a constant envelope equal to 1.

Let $Q$ be a probability measure on $\mathbb{R}^{p+1}$ with finite support. Fix $0<\varepsilon<1$.
By Theorem 2.6.7 from \citep{vdv}, for some universal constant $K$ we have that for $i=1,2$
\begin{align*}
N(\varepsilon,m^{i}(\mathcal{L}),\Vert . \Vert_{2,Q})\leq K (p+3) (16e)^{p+3} \varepsilon^{-2(p+2)}.
\end{align*}
Note that $m^{1}\circ \mathcal{L}+m^{2}\circ\mathcal{L}$ has constant envelope equal to 2. It is easy to show that
\begin{align*}
N(2 \varepsilon,m^{1}\circ\mathcal{L}+m^{2}\circ\mathcal{L},\Vert . \Vert_{2,Q}) &\leq N(\varepsilon/2 ,m^{1}\circ\mathcal{L},\Vert . \Vert_{2,Q}) N(\varepsilon/2,m^{2}\circ\mathcal{L},\Vert . \Vert_{2,Q})
\\ &\leq (K (p+3) (16e)^{p+3} \left(\varepsilon/2 \right)^{-2(p+2)})^{2}.
\end{align*}
Note that $\mathcal{H}$ has envelope $H(u,\mathbf{x})=2$ and that $\mathcal{H} \subset m^{1}\circ\mathcal{L}+ m^{2}\circ\mathcal{L}$.
Hence
\begin{align*}
N(\varepsilon \Vert H \Vert_{2,Q},\mathcal{H},\Vert . \Vert_{2,Q})\leq (K (p+3) (16e)^{p+3} \left(\varepsilon/2 \right)^{-2(p+2)})^{2}.
\end{align*}
Furthermore $\mathcal{H}\cup\lbrace 0 \rbrace$ also has envelope $H$. We can assume without loss of generality that $K>1$. Hence,
\begin{align*}
N(\varepsilon \Vert H \Vert_{2,Q},\mathcal{H}\cup \lbrace 0 \rbrace,\Vert . \Vert_{2,Q}) &\leq N( \varepsilon \Vert H \Vert_{2,Q},\mathcal{H},\Vert . \Vert_{2,Q})+1
\\ &\leq (K (p+3) (16e)^{p+3} \left(\varepsilon/2 \right)^{-2(p+2)})^{2} + 1 
\\ &\leq 2(K (p+3) (16e)^{p+3} \left(\varepsilon/2 \right)^{-2(p+2)})^{2}
\end{align*}
implies that
\begin{align*}
D(\varepsilon \Vert H \Vert_{2,Q},\mathcal{H}\cup \lbrace 0 \rbrace,\Vert . \Vert_{2,Q})& \leq N((\varepsilon/2) \Vert H \Vert_{2,Q},\mathcal{H}\cup \lbrace 0 \rbrace,\Vert . \Vert_{2,Q})
\\ &\leq D(\varepsilon)
\end{align*}
where
\begin{align*}
D(\varepsilon)= 2(K (p+3) (16e)^{p+3} \left(\varepsilon/4 \right)^{-2(p+2)})^{2}.
\end{align*}

It follows from Theorem \ref{Theo-Pollard}(i) that for some fixed $C_1>0$
\begin{align*}
\mathbb{E}\sup_{\mathcal{H}} \left\vert \frac{1}{\sqrt{n}} \sum\limits_{i=1}^{n}  \left( h(u_i,\mathbf{x}_{i}) - \mathbb{E}h(u,\mathbf{x}_{i})  \right) \right\vert &\leq \mathbb{E}\sup_{\mathcal{H} \cup \lbrace 0 \rbrace} \left\vert \frac{1}{\sqrt{n}} \sum\limits_{i=1}^{n} \left( h(u_i,\mathbf{x}_{i}) - \mathbb{E}h(u,\mathbf{x}_{i})  \right) \right\vert 
\\ &\leq C_1  \int\limits_{0}^{1} \left( \log D(\varepsilon) \right)^{1/2}.
\end{align*}
Note that $
\log D(\varepsilon) = \log 2 + 2\log(K) + 2 \log \left( p+3 \right) + 2\left( p+3 \right) \log (16e) + 4(p+2) \log \frac{4}{\varepsilon}
\leq C_2  p \left(1 - \log \varepsilon \right)$,
for some fixed $C_2>0$.
Hence
\begin{align*}
\mathbb{E}\sup_{\mathcal{H}} \left\vert \frac{1}{\sqrt{n}} \sum\limits_{i=1}^{n}  \left( h(u_i,\mathbf{x}_{i}) - \mathbb{E}h(u,\mathbf{x}_{i}) \right)  \right\vert \leq  \sqrt{p} C_1 \sqrt{C_2}  \int\limits_{0}^{1} \left(1-\log \varepsilon \right)^{1/2} d\varepsilon=\sqrt{p} C_3
\end{align*}
where $C_3>0$ is fixed.
Finally, the result follows from applying Markov's inequality and the fact that by assumption $p/n \rightarrow 0$.
\end{proof}

\begin{proof}[Proof of Lemma \ref{lemma-conv-sca}]
This follows from Theorem 3 of \citep{Davies}, replacing any appeals in the proof of that theorem to Lemma 2 of \citep{Davies} by appeals to Lemma \ref{lemma2-unif-davies}.
\end{proof}

The following lemma was needed in the proof of Theorem \ref{Theo-rate}.
\begin{lemma}
\label{lemma-orden1}
Assume [R2], [F0] and [X1] a) hold. Let $0<a<b$. 
For $\mathbf{x} \in \mathbb{R}^{p}$, consider the class of functions
$
\mathcal{H}=\left \lbrace h_{s}(u,\mathbf{x})=\psi_1 \left( u/s \right) \mathbf{x}: s\in [a,b] \right\rbrace.
$
Then, for some fixed constant $A>0$ that depends only on $a, b, \psi_1$ and the constant that appears in [X1] a),
\begin{align*}
\mathbb{E}\sup_{h\in\mathcal{H}} \left\Vert \frac{1}{\sqrt{n}} \sum\limits_{i=1}^{n}  h(u_i,\mathbf{x}_{i})  \right\Vert  \leq A \sqrt{p}.
\end{align*}
\end{lemma}

\begin{proof}
Let $
\mathcal{G}=\left \lbrace g_{s}(u,x)=\psi_1 \left( u/s \right) x: s\in [a,b] \right\rbrace.$
Fix $1\leq j \leq p$. Note that $\mathbb{E}g(u,x_{i,j})=0$ for $g\in\mathcal{G}$ and $i=1,\dots,n$.  Note also that $\mathcal{G}$ has envelope $G(u,x) = \Vert \psi_{1}\Vert_{\infty} \vert x \vert$ and that $
 (1/n)\sum_{i=1}^{n}\mathbb{E}G^{2}(u_i,x_{i,j})=(\Vert \psi_{1}\Vert_{\infty}^{2}/n)\sum_{i=1}^{n} x_{i,j}^{2}<\infty$.
Let $Q$ be a probability measure on $\mathbb{R}^{2}$ with finite support such that $\Vert G \Vert_{2,Q}>0$. This implies that $\Vert x \Vert_{2,Q}>0$

Let $\phi_1(t)=t \psi_{1}^{\prime}(t)$. By [R2], $\phi_1$ is bounded. Also, if $s_1 , s_2 \in [a,b]$, then by the Mean Value Theorem $
\vert g_{s_1}(u,x) - g_{s_2}(u,x)\vert \leq (\Vert \phi_1 \Vert_{\infty}   \vert x \vert \left\vert s_1 -s_2 \right\vert)/a.
$
Then, by Theorem 2.7.11 of \citep{vdv}, for all $\varepsilon>0$ the bracketing number of $\mathcal{G}$ satisfies
\begin{equation}
N_{[ \: ]}(2 \varepsilon \Vert \phi_1 \Vert_{\infty} \frac{1}{a} \Vert \vert x \vert\Vert_{2,Q} , \mathcal{G}, \Vert . \Vert_{2,Q}) \leq N(\varepsilon, [a,b], \vert \: . \: \vert).
\label{cota-bracket-lip-ord1}
\end{equation}
Note that for some constant $C_1$ that depends only on $a$ and $b$, for all $\varepsilon>0$
\begin{equation}
N(\varepsilon, [a,b], \vert \: . \: \vert) \leq  \frac{C_1}{\varepsilon} + 1.
\label{cota-covering-inter-ord1}
\end{equation}
Fix $0<\varepsilon<1$. It follows from \eqref{cota-bracket-lip-ord1} and \eqref{cota-covering-inter-ord1} that
\begin{align*}
N(\varepsilon \Vert G \Vert_{2,Q}, \mathcal{G}, \Vert . \Vert_{2,Q}) & = N(\varepsilon \Vert \psi_{1}\Vert_{\infty} \Vert \vert x \vert \Vert_{2,Q}, \mathcal{G}, \Vert . \Vert_{2,Q}) 
\\ &\leq N_{[\:]}(2 \varepsilon \Vert \psi_{1}\Vert_{\infty} \Vert \vert x \vert \Vert_{2,Q}, \mathcal{G}, \Vert . \Vert_{2,Q}) 
\\ &\leq N( \frac{a \varepsilon  \Vert \psi_1 \Vert_{\infty}}{\Vert \phi_1 \Vert_{\infty} }, [a,b], \vert \: . \: \vert) \\ &\leq \frac{C_1 \Vert \phi_1 \Vert_{\infty}}{a \varepsilon \Vert \psi_1 \Vert_{\infty}} +1 = \frac{C_2}{\varepsilon} +1.
\end{align*}
Note that $\mathcal{G} \cup \lbrace 0 \rbrace$ has envelope $G$, $\mathbb{E}g(u,x_{i,j})=0$ for $g\in\mathcal{G}\cup \lbrace 0 \rbrace$ and $i=1,\dots,n$, and that
\begin{align*}
N(\varepsilon \Vert G \Vert_{2,Q}, \mathcal{G}\cup \lbrace 0 \rbrace, \Vert . \Vert_{2,Q}) \leq N(\varepsilon \Vert G \Vert_{2,Q}, \mathcal{G}, \Vert . \Vert_{2,Q}) + 1 \leq \frac{C_2}{\varepsilon} +2.
\end{align*}
Thus
\begin{align*}
D(\varepsilon \Vert G \Vert_{2,Q}, \mathcal{G} \cup \lbrace 0 \rbrace, \Vert . \Vert_{2,Q}) \leq N(\varepsilon \Vert G \Vert_{2,Q} /2, \mathcal{G} \cup \lbrace 0 \rbrace, \Vert . \Vert_{2,Q}) \leq \frac{2C_2}{\varepsilon} +2 .
\end{align*}
Let $D(\varepsilon)= 2 C_2/\varepsilon +2$. Then by Theorem \eqref{Theo-Pollard}(ii), for some fixed $C_3>0$
\begin{align}
\mathbb{E}\sup_{\mathcal{G}} \left\vert \frac{1}{\sqrt{n}} \sum\limits_{i=1}^{n}  g(u_i,x_{i,j})  \right\vert^{2} &\leq \mathbb{E}\sup_{\mathcal{G}\cup \lbrace 0 \rbrace} \left\vert \frac{1}{\sqrt{n}} \sum\limits_{i=1}^{n}  g(u_i,x_{i,j})  \right\vert^{2} \nonumber \\ &\leq C_3 \left( \frac{1}{n} \sum\limits_{i=1}^{n} x_{i,j}^{2} \right)\left(\int\limits_{0}^{1} \left( \log\left(\frac{2 C_2}{\varepsilon} +2\right) \right)^{1/2}d\varepsilon \right)^{2} .
\label{cota-coordenada}
\end{align}
Note that \eqref{cota-coordenada} holds for all $1\leq j \leq p$.
Then, by \eqref{cota-coordenada} and [X1] a), for some fixed $C_4>0$
\begin{align*}
\mathbb{E}\sup_{\mathcal{H}} \left\Vert \frac{1}{\sqrt{n}} \sum\limits_{i=1}^{n}  h(u_i,\mathbf{x}_{i})  \right\Vert^{2} = \mathbb{E}\sup_{s\in[a,b]} \sum\limits_{j=1}^{p} \left\vert \frac{1}{\sqrt{n}}\sum\limits_{i=1}^{n} \psi_{1}\left( \frac{u_i}{s}\right) x_{i,j} \right\vert^2
\\ \leq \sum\limits_{j=1}^{p} \mathbb{E}  \sup_{s\in[a,b]} \left\vert \frac{1}{\sqrt{n}}\sum\limits_{i=1}^{n} \psi_{1}\left( \frac{u_i}{s}\right) x_{i,j} \right\vert^2  
\\ \leq \sum\limits_{j=1}^{p}C_3 \left( \frac{1}{n} \sum\limits_{i=1}^{n} x_{i,j}^{2} \right)  \left(\int\limits_{0}^{1} \left( \log\left(\frac{2 C_2}{\varepsilon} +2\right) \right)^{1/2}d\varepsilon \right)^{2} \leq C_4 p.
\end{align*}
The result now follows from applying Jensen's inequality.
\end{proof}

The following lemma, which is a very simple adaptation of Lemma 3.1 of \citep{Portnoy}, was needed to obtain the rate of consistency of the estimators. Define 
\begin{align*}
H_{i}(\mathbf{x}_{i}^{T}\bbet ) = \inf\left\lbrace \psi_{1}^{'}\left( \frac{u_i-v}{s_0} \right) : \vert v \vert \leq \vert \mathbf{x}_{i}^{T} \bbet \vert \right\rbrace.
\end{align*}

\begin{lemma}
\label{lemma-mod-portnoy}
Assume [R2], [F0], [X1], [X2], [X4] and [X5] hold. Assume $(p \log n)/n \rightarrow 0$. Then there exists $a^{*}>0$ and $\delta>0$ such that
\begin{equation}
\mathbb{P}\left( \inf \left\lbrace \sum\limits_{i=1}^{n} (\mathbf{x}_{i}^{T}\mathbf{z})^{2} H_{i}^{n}(\mathbf{x}_{i}^{T} \bbet)  : \Vert \mathbf{z} \Vert =1, \Vert \bbet \Vert \leq \delta\right\rbrace \geq a^{*}n \right)  \rightarrow 1.
\label{tesis-mod-portnoy}
\end{equation}
\end{lemma}

\begin{proof}
Note that
\begin{align*}
\frac{1}{n}\sum\limits_{i=1}^{n} (\mathbf{x}_{i}^{T}\mathbf{z})^{2} H_{i}^{n}(\mathbf{x}_{i}^{T} \bbet)&=\frac{1}{n} \sum\limits_{i=1}^{n} (\mathbf{x}_{i}^{T}\mathbf{z})^{2} H_{i}(\mathbf{x}_{i}^{T} \bbet) 
\\ &+\frac{1}{n} \sum\limits_{i=1}^{n} (\mathbf{x}_{i}^{T}\mathbf{z})^{2} \left(H_{i}^{n}(\mathbf{x}_{i}^{T} \bbet) - H_{i}(\mathbf{x}_{i}^{T} \bbet) \right).
\end{align*}
Hence, for any $\delta>0$
\begin{align*}
\inf_{\Vert \mathbf{z} \Vert =1, \Vert \bbet  \Vert \leq \delta} \frac{1}{n}\sum\limits_{i=1}^{n} (\mathbf{x}_{i}^{T}\mathbf{z})^{2} H_{i}^{n}(\mathbf{x}_{i}^{T} \bbet) &\geq  \inf_{\Vert \mathbf{z} \Vert =1, \Vert \bbet \Vert \leq \delta} \frac{1}{n} \sum\limits_{i=1}^{n} (\mathbf{x}_{i}^{T}\mathbf{z})^{2} H_{i}(\mathbf{x}_{i}^{T} \bbet) \\ &+ \inf_{\Vert \mathbf{z} \Vert =1, \Vert \bbet \Vert \leq \delta}  \frac{1}{n} \sum\limits_{i=1}^{n} (\mathbf{x}_{i}^{T}\mathbf{z})^{2} \left(H_{i}^{n}(\mathbf{x}_{i}^{T} \bbet) - H_{i}(\mathbf{x}_{i}^{T} \bbet) \right).
\end{align*}
By Lemma 3.1 of \citep{Portnoy}, \eqref{tesis-mod-portnoy} holds when $H_{i}^{n}$ is replaced by $H_{i}$. Hence, for some $a^{*}>0$ and $\delta>0$, for sufficiently large $n$, with arbitrarily high probability
\begin{equation}
\inf_{\Vert \mathbf{z} \Vert =1, \Vert \bbet  \Vert \leq \delta} \frac{1}{n} \sum\limits_{i=1}^{n} (\mathbf{x}_{i}^{T}\mathbf{z})^{2} H_{i}(\mathbf{x}_{i}^{T} \bbet) \geq a^{*}.
\nonumber
\end{equation}
We will show that 
\begin{align*}
\sup_{\Vert \mathbf{z} \Vert =1, \Vert \bbet \Vert \leq \delta}  \left\vert \frac{1}{n} \sum\limits_{i=1}^{n} (\mathbf{x}_{i}^{T}\mathbf{z})^{2} \left(H_{i}^{n}(\mathbf{x}_{i}^{T} \bbet) - H_{i}(\mathbf{x}_{i}^{T} \bbet) \right) \right\vert \cp 0.
\end{align*}
Fix $i\leq n$, $\mathbf{z}$ with $\Vert \mathbf{z} \Vert =1$, and $\bbet$ with $\Vert \bbet \Vert \leq \delta$. We will bound $\left\vert H_{i}^{n}(\mathbf{x}_{i}^{T} \bbet) - H_{i}(\mathbf{x}_{i}^{T} \bbet) \right\vert$. Assume $ H_{i}^{n}(\mathbf{x}_{i}^{T} \bbet) \geq H_{i}(\mathbf{x}_{i}^{T} \bbet)$. By [R2], $H_{i}(\mathbf{x}_{i}^{T} \bbet)= \psi_{1}^{\prime} \left( (u_i - v_{i}^{*})/s_0 \right)$ for some $v_{i}^{*}$ with $\vert v_{i}^{*} \vert \leq \vert \mathbf{x}_{i}^{T} \bbet \vert$. Then
\begin{align*}
\left\vert H_{i}^{n}(\mathbf{x}_{i}^{T} \bbet) - H_{i}(\mathbf{x}_{i}^{T} \bbet) \right\vert &= H_{i}^{n}(\mathbf{x}_{i}^{T} \bbet) - H_{i}(\mathbf{x}_{i}^{T} \bbet) 
\\&= H_{i}^{n}(\mathbf{x}_{i}^{T} \bbet) - \psi_{1}^{\prime} \left( \frac{u_i - v_{i}^{*}}{s_0} \right) \\ &\leq \psi_{1}^{\prime} \left( \frac{u_i - v_{i}^{*}}{s_n} \right) - \psi_{1}^{\prime} \left( \frac{u_i - v_{i}^{*}}{s_0} \right) \\ &\leq \left\vert \psi_{1}^{\prime} \left( \frac{u_i - v_{i}^{*}}{s_n} \right) - \psi_{1}^{\prime} \left( \frac{u_i - v_{i}^{*}}{s_0} \right) \right \vert.
\end{align*}
Note that by [R2], $\phi(t)=\psi_{1}^{\prime\prime}(t)t$ is bounded. Applying the Mean Value Theorem we get that
\begin{align*}
\left\vert \psi_{1}^{\prime} \left( \frac{u_i - v_{i}^{*}}{s_n} \right) - \psi_{1}^{\prime} \left( \frac{u_i - v_{i}^{*}}{s_0} \right) \right \vert &= \left\vert \psi''_{1} \left( \frac{u_i - v_{i}^{*}}{s_{i,n}^{*}}\right) \left( \frac{u_i - v_{i}^{*}}{s_{i,n}^{*}} \right) \right\vert   \left\vert \frac{s_n-s_0}{s_{i,n}^{*}} \right\vert \\ &\leq \Vert \phi\Vert_{\infty} \left\vert \frac{s_n-s_0}{s_{i,n}^{*}} \right\vert.
\end{align*}
where $s_{i,n}^{*}$ is such that $\vert s_{i,n}^{*} - s_0\vert \leq \vert s_n - s_0 \vert$. Note that $s_{i,n}^{*}$ may depend on $\bbet$, say $s_{i,n}^{*}=s_{i,n}^{*}(\bbet)$. The same type of argument can be used to show that an analogous bound holds when $ H_{i}^{n}(\mathbf{x}_{i}^{T} \bbet) \leq H_{i}(\mathbf{x}_{i}^{T} \bbet)$. 

Note that since $s_n \cp s_0$, we have that $\sup_{\Vert \bbet \Vert \leq \delta } \max_{i} \vert s_{i,n}^{*}(\bbet) - s_0 \vert \leq \vert s_n -s_0 \vert \cp 0$. Then
\begin{align*}
&\sup_{\Vert \mathbf{z} \Vert =1, \Vert \bbet  \Vert \leq \delta}  \left\vert \frac{1}{n} \sum\limits_{i=1}^{n} (\mathbf{x}_{i}^{T}\mathbf{z})^{2} \left(H_{i}^{n}(\mathbf{x}_{i}^{T} \bbet) - H_{i}(\mathbf{x}_{i}^{T} \bbet) \right) \right\vert \\ &\leq 
\sup_{\Vert \mathbf{z} \Vert =1, \Vert \bbet \Vert \leq \delta} \frac{1}{n} \sum\limits_{i=1}^{n} (\mathbf{x}_{i}^{T}\mathbf{z})^{2} \left\vert H_{i}^{n}(\mathbf{x}_{i}^{T} \bbet) - H_{i}(\mathbf{x}_{i}^{T} \bbet) \right\vert  \\ &\leq \sup_{\Vert \mathbf{z} \Vert =1, \Vert \bbet \Vert \leq \delta}\Vert \phi\Vert_{\infty} \max_{i}  \frac{1}{\left\vert s_{i,n}^{*}(\bbet) \right\vert}  \left\vert s_n-s_0 \right\vert \frac{1}{n} \sum\limits_{i=1}^{n} (\mathbf{x}_{i}^{T}\mathbf{z})^{2} \\ &\leq \sup_{\Vert \bbet \Vert \leq \delta} \Vert \phi\Vert_{\infty} \max_{i} \frac{1}{\vert s_{i,n}^{*}(\bbet) \vert }\left\vert s_n-s_0 \right\vert \tau \cp 0.
\end{align*}
It follows that for sufficiently large $n$, with arbitrarily high probability,
\begin{equation}
\inf \left\lbrace \sum\limits_{i=1}^{n} (\mathbf{x}_{i}^{T}\mathbf{z})^{2} H_{i}^{n}(\mathbf{x}_{i}^{T} \bbet)  : \Vert \mathbf{z} \Vert =1, \Vert \bbet \Vert \leq \delta\right\rbrace \geq n(a^{*} - a^{*}/2),
\nonumber
\end{equation}
and so the lemma is proven.
\end{proof}

The following lemma was needed in the proof of Theorem \ref{Theo-Orac}.

\begin{lemma}
\label{lemma-tight}
Assume [R2], [F0], [X1], [X2], [X3] and [X6]  hold. Let $\mathbf{a}_n\in\mathbb{R}^{p}$, $\Vert \mathbf{a}_{n}%
\Vert = 1$. Let $r_{n}^{2}= 
\mathbf{a}_{n}^{T}\boldsymbol{\Sigma}_{n}^{-1}\mathbf{a}_{n}$. Then
\begin{itemize}
\item[a)] \begin{equation}
\frac{1}{\sqrt{n}}\sum\limits_{i=1}^{n} \left(\psi_1 \left( \frac{u_i}{s_n}
\right) - \psi_1 \left( \frac{u_i}{s_0} \right)\right) \left(\mathbf{a}_{n}^{T} \boldsymbol{\Sigma}_{n}^{-1} \mathbf{x%
}_{i}\right) \overset{P}{\rightarrow} 0.  \notag
\end{equation}

\item[b)] \begin{equation}
\frac{1}{r_n \sqrt{n}}\sum\limits_{i=1}^{n} \psi_1 \left( \frac{u_i}{s_0} \right) \mathbf{a}_{n}^{T} \boldsymbol{\Sigma}_{n}^{-1}\mathbf{x%
}_{i} \overset{%
d}{\rightarrow} N \left( 0,\mathbb{E} \psi_{1}^{2} \left(\frac{u}{s_0}\right) \right).
\nonumber
\end{equation}
\end{itemize}
\end{lemma}

\begin{proof}
We first prove a). For $t\in[0,1]$ let 
\begin{align*}
G_n(t) = \frac{1}{\sqrt{n}}\sum_{i=1}^{n} \psi_1 \left( \frac{u_i}{
0.5s_0+ts_0} \right) \mathbf{a}_{n}^{T} \boldsymbol{\Sigma}_{n}^{-1}\mathbf{x}_{i}.
\end{align*}
Since by assumption $s_n \overset{P}{\rightarrow} s_0$, it suffices to show that $%
(G_n)_n$ is a tight sequence in $C[0,1]$. By Theorem 12.3 of \citep{Billingsley}, it suffices to show that

\begin{enumerate}
\item[(i)] $G_n(0)$ is tight

\item[(ii)] There exists $\gamma \geq 0$, $\alpha>1$ and a nondecreasing, continuous
function $f$ on $[0,1]$, such that for any $0\leq t_1 \leq t_2 \leq 1$ and
any $\lambda>0$ we have 
\begin{equation}
\mathbb{P} \left( \vert G_n(t_2) - G_n(t_1) \vert \geq \lambda \right) \leq 
\frac{1}{\lambda^{\gamma}} \left( f(t_2) - f(t_1) \right)^{\alpha} \text{ for all } n.
\nonumber
\end{equation}
\end{enumerate}

We first prove (i). Let $
h_{n}^{2}= \mathbb{E} \psi_{1}^{2} \left(u/(0.5s_0)\right) 
\mathbf{a}_{n}^{T}\boldsymbol{\Sigma}_{n}^{-1}\mathbf{a}_{n}$. 
By [X1], [X3] and Lemma \ref{lemma:cond_davies}, $\inf_{n} \rho_{1,n}>0$. This together with [X2] implies that $h_n$ and $1/h_n$ are bounded. Note that since $
\psi_1$ is odd and the errors have a symmetric distribution, $\mathbb{E} \psi_1 \left( u/(0.5s_0) \right) = 0$.
Also, \begin{align*}
\sum\limits_{i=1}^{n} \mathbb{E} \left( \frac{1}{\sqrt{n}} \psi_1 \left( \frac{u}{0.5s_0} \right) \mathbf{a}_{n}^{T} \boldsymbol{\Sigma}_{n}^{-1}\mathbf{x}_{i} \right)^2 =
h_{n}^{2}.
\end{align*}
Note that by [X6] $\max_{i\leq n} (\mathbf{a}_{n}^{T} \boldsymbol{\Sigma}_{n}^{-1}\mathbf{x}_{i}) / (\sqrt{n} h_n)
\rightarrow 0$. Then for any fixed $\varepsilon>0$, 
\begin{align*}
\sum\limits_{i=1}^{n}\mathbb{E}\left( \frac{1}{\sqrt{n} h_n} \psi_1 \left( 
\frac{u}{0.5s_0} \right) \mathbf{a}_{n}^{T} \boldsymbol{\Sigma}_{n}^{-1}\mathbf{x}_{i} \right)^{2} I \left\lbrace  \left\vert \psi_1 \left( 
\frac{u}{0.5s_0} \right)(\mathbf{a}_{n}^{T} \boldsymbol{\Sigma}_{n}^{-1}\mathbf{x}_{i}) / (\sqrt{n} h_n)\right\vert>\varepsilon \right\rbrace
\\ \rightarrow 0.
\end{align*}
Hence, by the Lindberg-Feller Theorem, $G_n(0) /h_n \overset{%
d}{\rightarrow} N(0,1)$ and (i) follows. Note that roughly the same argument proves b).

Now, we prove (ii). By Tchebyshev's inequality, it suffices to show that
there exists $K>0$ such that for all $t_1 , t_2$ in $[0,1]$, $\mathbb{E}
(G_n(t_1) -G_n(t_2))^{2} \leq K (t_2-t_1)^2$ for all $n$. Let 
\begin{equation}
\Delta_i (t_1,t_2) = \psi_1 \left( \frac{u_i}{
0.5s_0+t_1 s_0} \right)  -\psi_1 \left( \frac{u_i}{
0.5s_0+t_2 s_0} \right).
\nonumber
\end{equation}
Note that $\mathbb{E}\Delta_i(t_1,t_2)=0$ for all $t_1, t_2$ and $i$.
Using the independence of $u_1, \dots, u_n$, we get
\begin{align}
\mathbb{E} (G_n(t_1) -G_n(t_2))^{2} &= \frac{1}{n}\sum_{i,j} \mathbb{E}   \Delta_i(t_1,t_2)\Delta_j(t_1,t_2)  (\mathbf{a}_{n}^{T} \boldsymbol{\Sigma}_{n}^{-1}\mathbf{x}_{i}) (\mathbf{a}_{n}^{T} \boldsymbol{\Sigma}_{n}^{-1}\mathbf{x}_{i})\nonumber\\ &= \mathbb{E}\Delta_{1}(t_1,t_2)^{2} \frac{1}{n} \sum\limits_{i=1}^{n} (\mathbf{a}_{n}^{T} \boldsymbol{\Sigma}_{n}^{-1}\mathbf{x}_{i})^{2}
\nonumber\\ &=
\mathbb{E} \left( \psi_1 \left( \frac{u}{0.5s_0+t_1 s_0} \right) -
\psi_1 \left( \frac{u}{0.5s_0+t_2 s_0} \right) \right)^{2} \mathbf{a}_{n}%
^{T} \boldsymbol{\Sigma}_{n}^{-1} \mathbf{a}_{n}
\label{desig-fin-tight}.
\end{align}
Let $\phi_1(t)=\psi_{1}^{\prime}(t)t$. By [R2] $\phi_1$ is bounded. Applying the Mean Value Theorem we get that
\begin{align*}
&\left\vert \psi_1 \left( \frac{u}{0.5s_0+t_1 s_0} \right) - \psi_1 \left( \frac{u}{0.5s_0+t_2 s_0} \right) \right\vert \\ &= \left\vert \psi_{1}^{\prime} \left( \frac{u}{0.5s_0+t^{*} s_0} \right) \left( \frac{u}{0.5s_0+t^{*} s_0} \right) \left( \frac{s_0}{0.5s_0+t^{*} s_0}\right) \left( t_1 - t_2\right) \right\vert \\ &\leq 2 \Vert \phi_{1}\Vert_{\infty} \left\vert  \left( t_1 -t_2 \right) \right\vert,
\end{align*}
where $t^{*}$ lies between $t_1$ and $t_2$.

Hence, for some fixed constant $C>0$
\begin{align*}
\mathbb{E} \left( \psi_1 \left( \frac{u}{0.5s_0+t_1 s_0} \right) -
\psi_1 \left( \frac{u}{0.5s_0+t_2 s_0} \right) \right)^{2} \leq C (t_2-t_1)^{2}.
\end{align*}
Hence, since $\inf_{n} \rho_{1,n}>0$, from \eqref{desig-fin-tight} it follows that (ii) holds and thus the lemma is proven.
\end{proof}

The following lemma was needed in the proof of Theorem \ref{Theo-Orac}. Its proof is very similar to that of Lemma \ref{lemma-orden1} and for this reason it is ommitted.
\begin{lemma}
\label{lemma-matriz}
Assume [R2], [F0] and [X1] a) hold. Let $0<a<b$. Then for some fixed constant $A>0$ that depends only on $a,b$, $\psi_{1}^{\prime}$ and the constant that appears in [X1] a),
\begin{align*}
\mathbb{E}\sup_{s\in [a,b]} \left\Vert (\frac{1}{\sqrt{n}}) \sum\limits_{i=1}^{n} \left( \psi_{1}^{\prime}\left( \frac{u_{i}}{s}\right) - \mathbb{E}\psi_{1}^{\prime}\left(\frac{u}{s}\right)\right)  \mathbf{x}_{i}\mathbf{x}_{i}^{T} \right\Vert_{F} \leq A \sqrt{p} \max_{i\leq n} \Vert \mathbf{x}_{i}\Vert,
\end{align*}
where $\Vert . \Vert_{F}$ is the Frobenius norm.
\end{lemma}

\section*{Acknowledgments}
Work supported in part by Grant PIP 112-201101-00339 from CONICET and by a CONICET Doctoral Fellowship.
This paper is based on the author's Ph.D. dissertation dissertation at the University of Buenos Aires. The author would like to express his gratitude to Graciela Boente, Daniela Rodriguez, Mariela Sued and Victor J. Yohai for their support, encouragement and helpful suggestions.

\bibliography{div}

\begin{thebibliography}{10}

\bibitem{BaiI}
Z.~D. Bai and Y.~Wu.
\newblock Limiting behavior of {M}-estimators of regression coefficients in
  high dimensional linear models {I}. scale dependent case.
\newblock {\em J. Multivar. Anal.}, 51(2):211--239, 1994.

\bibitem{BaiII}
Z.~D. Bai and Y.~Wu.
\newblock Limiting behavior of {M}-estimators of regression-coefficients in
  high dimensional linear models {II}. scale-invariant case.
\newblock {\em J. Multivar. Anal.}, 51(2):240--251, 1994.

\bibitem{Billingsley}
P.~Billingsley.
\newblock {\em Convergence of Probability Measures}.
\newblock Wiley, 1968.

\bibitem{Davies}
Laurie Davies.
\newblock The asymptotics of {S}-estimators in the linear regression model.
\newblock {\em Ann. Statist.}, 18(4):1651--1675, 12 1990.

\bibitem{Donoho-Breakdown}
D.~L. {Donoho} and P.~J. {Huber}.
\newblock The notion of breakdown point.
\newblock In P.~J. Bickel, K.~A. Doksum, and Jr. J.~L.~Hodges, editors, {\em A
  Festschrift for Erich L. Lehmann}, pages 157--185. Wadsworth, 1983.

\bibitem{Donoho-Montanari}
D.~L {Donoho} and A.~{Montanari}.
\newblock High dimensional robust {M}-estimation: Asymptotic variance via
  approximate message passing.
\newblock {\em Probab. Theory Related Fields}, pages 1--35, 2015.

\bibitem{Donoho-MontanariH}
D.~L. {Donoho} and A.~{Montanari}.
\newblock {Variance Breakdown of {Huber} {(M)}-estimators: $n/p \rightarrow m
  \in (1,\infty)$}.
\newblock {\em ArXiv e-prints}, March 2015.
\newblock Available at \url{https://arxiv.org/abs/1503.02106}.

\bibitem{Karoui2}
N.~{El Karoui}.
\newblock {Asymptotic Behavior of Unregularized and Ridge-regularized
  High-dimensional Robust Regression Estimators : Rigorous Results}.
\newblock {\em ArXiv e-prints}, November 2013.
\newblock Available at \url{https://arxiv.org/abs/1311.2445}.

\bibitem{Karoui}
Noureddine El~Karoui, Derek Bean, Peter~J. Bickel, Chinghway Lim, and Bin Yu.
\newblock On robust regression with high-dimensional predictors.
\newblock {\em Proc. Natl. Acad. Sci. U.S.A.}, 110(36):14557--14562, 2013.

\bibitem{Fasano}
María~V. Fasano, Ricardo~A. Maronna, Mariela Sued, and Víctor~J. Yohai.
\newblock Continuity and differentiability of regression {M} functionals.
\newblock {\em Bernoulli}, 18(4):1284--1309, 11 2012.

\bibitem{HampelLibro}
Frank~R Hampel, Elvezio~M Ronchetti, Peter~J Rousseeuw, and Werner~A Stahel.
\newblock {\em Robust statistics: the approach based on influence functions}.
\newblock John Wiley \& Sons, 1986.

\bibitem{HE2000120}
Xuming He and Qi-Man Shao.
\newblock On parameters of increasing dimensions.
\newblock {\em J. Multivar. Anal.}, 73(1):120 -- 135, 2000.

\bibitem{Hossjer}
O.~Hossjer.
\newblock On the optimality of {S}-estimators.
\newblock {\em Statist. Probab. Lett.}, 14(5):413 -- 419, 1992.

\bibitem{Huber64}
P.~J. Huber.
\newblock Robust estimation of a location parameter.
\newblock {\em Ann. Math. Statist.}, 35(1):73--101, 03 1964.

\bibitem{Huber73}
P.~J. Huber.
\newblock Robust regression: Asymptotics, conjectures and {Monte Carlo}.
\newblock {\em Ann. Statist.}, 1(5):799--821, 09 1973.

\bibitem{Huber81}
P.~J. Huber.
\newblock {\em Robust Statistics}.
\newblock Wiley, 1981.

\bibitem{Kosorok}
M.~Kosorok.
\newblock {\em Introduction to Empirical Processes and Semiparametric
  Inference}.
\newblock Springer, 2008.

\bibitem{Mammen88}
E.~Mammen.
\newblock Asymptotics with increasing dimension for robust regression with
  applications to the bootstrap.
\newblock {\em Ann. Statist.}, 17(1):382--400, 03 1989.

\bibitem{Libro}
R.~A. Maronna, D.~R. Martin, and V.~J. Yohai.
\newblock {\em Robust Statistics: Theory and Methods}.
\newblock Wiley, 2006.

\bibitem{Ritov}
D.~{Nevo} and Y.~{Ritov}.
\newblock {On Bayesian Robust Regression with Diverging Number of Predictors}.
\newblock {\em ArXiv e-prints}, July 2015.
\newblock Available at \url{http://arxiv.org/pdf/1507.02074v2.pdf}.

\bibitem{Pollard89}
D.~Pollard.
\newblock Asymptotics via empirical processes.
\newblock {\em Statist. Sci.}, 4(4):341--354, 11 1989.

\bibitem{Portnoy}
S.~Portnoy.
\newblock Asymptotic behavior of {M}-estimators of $p$ regression parameters
  when $p^2/n$ is large. {I}. consistency.
\newblock {\em Ann. Statist.}, 12(4):1298--1309, 12 1984.

\bibitem{Portnoy85}
S.~Portnoy.
\newblock Asymptotic behavior of {M}-estimators of $p$ regression parameters
  when $p^2 / n$ is large. {II}. normal approximation.
\newblock {\em Ann. Statist.}, 13(4):1403--1417, 12 1985.

\bibitem{S}
P.~J. Rousseeuw and V.~J. Yohai.
\newblock Robust regression by means of {S}-estimators.
\newblock In J{\"u}rgen Franke, Wolfgang H{\"a}rdle, and Douglas Martin,
  editors, {\em Robust and Nonlinear Time Series Analysis}, pages 256--272.
  Springer US, 1984.

\bibitem{vdv}
A.~W. van~der vaart and J.~Wellner.
\newblock {\em Weak Convergence and Empirical Processes: With Applications to
  Statistics}.
\newblock Springer-Verlag New York, 1996.

\bibitem{Welsh}
A.~H. Welsh.
\newblock On {M}-processes and {M}-estimation.
\newblock {\em Ann. Statist.}, 17(1):337--361, 03 1989.

\bibitem{MM85}
V.~J. Yohai.
\newblock {High Breakdown Point and High Efficiency Robust Estimates for
  Regression}.
\newblock Technical Report~66, University of Washington, 1985.
\newblock Available at
  \url{http://www.stat.washington.edu/research/reports/1985/tr066.pdf}.

\bibitem{MM87}
V.~J. Yohai.
\newblock High breakdown-point and high efficiency robust estimates for
  regression.
\newblock {\em Ann. Statist.}, 15(2):642--656, 06 1987.

\bibitem{Maronna79}
V.~J. {Yohai} and R.~A. {Maronna}.
\newblock Asymptotic behavior of {M}-estimators for the linear model.
\newblock {\em Ann. Statist.}, 7(2):258--268, 03 1979.

\end{thebibliography}

\Addresses
\end{document}